\documentclass{siamart171218}
\usepackage{amsfonts}
\usepackage{epstopdf}
\usepackage{algorithmic}
\ifpdf
  \DeclareGraphicsExtensions{.eps,.pdf,.png,.jpg}
\else
  \DeclareGraphicsExtensions{.eps}
\fi

\usepackage{amsopn}
\usepackage[english]{babel}
\usepackage{a4wide}
\usepackage{mathtools}      
\usepackage{amstext, amsmath, amssymb}
\usepackage{nccmath}
\usepackage{graphicx}
\usepackage{latexsym}
\usepackage{booktabs}
\usepackage{url}
\usepackage{import}
\usepackage{fancyvrb}
\usepackage{stmaryrd}
\usepackage{pdfsync}
\usepackage{algorithm}
\usepackage{tikz}
\usepackage{pgfplots}
\usepackage{pgfplotstable}
\pgfplotsset{compat=newest}
\usepackage{subcaption}
\usepackage[nice]{nicefrac}
\usepackage{tabularx}
\usepackage{todonotes}
\usepackage{enumitem}
\usepackage[numbers]{natbib}

\usepackage{array}	
\usepackage{verbatim}
\usepackage[normalem]{ulem}

\ifpdf
\hypersetup{
                pdftitle={Stabilized cut discontinuous Galerkin methods for advection-reaction problems},
                pdfauthor={C. G\"urkan, S. Sticko, and A. Massing}
              }
\fi

\newcommand{\RR}{\mathbb{R}}

\newcommand{\PP}{\mathbb{P}}

\newcommand{\st}{\;|\;}

\newcommand{\bft}{\boldsymbol{t}}

\newcommand{\bfbeta}{\boldsymbol{\bfbeta}}

\newcommand{\mcP}{\mathcal{P}}

\newcommand{\mcL}{\mathcal{L}}
\newcommand{\mcF}{\mathcal{F}}

\newcommand{\mcA}{\mathcal{A}}
\newcommand{\mcM}{\mathcal{M}}

\newcommand{\mcT}{\mathcal{T}}

\newcommand{\mcO}{\mathcal{O}}

\newcommand{\tn}{|\mspace{-1mu}|\mspace{-1mu}|}

\newcommand{\jump}[1]{[#1]}
\newcommand{\mean}[1]{\{{#1} \}}
\newcommand{\avg}[1]{\{{#1} \}}

\newcommand{\tnup}{\tn_{\mathrm{up}}}
\newcommand{\tnuph}{\tn_{\mathrm{up},h}}
\newcommand{\tnsdhast}{\tn_{\mathrm{sd},h,\ast}}
\newcommand{\tnupast}{\tn_{\mathrm{up},\ast}}
\newcommand{\tnsd}{\tn_{\mathrm{sd}}}
\newcommand{\tnsdh}{\tn_{\mathrm{sd},h}}
\newcommand{\tnsdast}{\tn_{\mathrm{sd},\ast}}
\newcommand{\tncf}{\tn_{\mathrm{cf}}}

\newcommand{\bnabla}{b \cdot \nabla}
\newcommand{\bhnabla}{b_h^0 \cdot \nabla}

\newcommand{\foralls}{\forall\,}

\newcommand{\partialbar}{{\partial}}

\newcommand{\evaluated}[2]{\left.#1\right|_{#2}}

\newcommand{\onehalf}{\nicefrac{1}{2}}
\newcommand{\donehalf}{\dfrac{1}{2}}

\newcommand{\trootonehalf}{\tfrac{1}{\sqrt{2}}}
\newcommand{\threehalf}{\nicefrac{3}{2}}

\DeclareMathOperator{\nDiv}{\nabla \cdot}

\DeclareMathOperator{\diam}{diam}
\DeclareMathOperator*{\essinf}{ess\, inf}
\DeclareMathOperator{\atantwo}{atan_2}
\DeclareMathOperator{\Id}{Id}

\DefineVerbatimEnvironment{code}{Verbatim}{frame=single,rulecolor=\color{blue}}

\numberwithin{equation}{section}
\numberwithin{figure}{section}
\numberwithin{table}{section}

\newsiamremark{remark}{Remark}
\newsiamremark{hypothesis}{Hypothesis}
\crefname{hypothesis}{Hypothesis}{Hypotheses}
\newsiamthm{claim}{Claim}
\newtheorem{assumption}{Assumption}

\headers{Stabilized CutDG methods for advection-reaction problems}{C. G\"urkan, S. Sticko, and A. Massing}

\title{Stabilized cut discontinuous Galerkin methods for advection-reaction problems\thanks{Submitted to the editors \today.
    \funding{This work was funded by the Kempe Foundation under Postdoc Scholarship JCK-1612,
by the Swedish Research Council under Starting Grant 2017-05038, and by the Swedish Research Programme Essence.}}}

\author{Ceren G\"urkan\thanks{Department of Civil Engineering, Kadir Has University, Cibali, 34083 Fatih,  Istanbul, Turkey (\email{ceren.gurkan@khas.edu.tr}).}
\and Simon Sticko\thanks{Department of Mathematics and Mathematical Statistics, Ume{\aa} University, SE-90187 Ume{\aa}, Sweden,
  (\email{andre.masssing@umu.se}, \email{simon.sticko@umu.se}).}
\and Andr\'e Massing\footnotemark[3] 
\thanks{
Department of Mathematical Sciences, Norwegian University of Science and Technology, NO-7491 Trondheim, Norway, (\email{andre.massing@ntnu.no})}
}

\begin{document}
\maketitle
  
\begin{abstract} 
  We develop novel stabilized cut discontinuous
  Galerkin (CutDG) methods for advection-reaction problems.
  The domain of interest is embedded into a structured, unfitted
  background mesh in $\mathbb{R}^d$ where the domain
  boundary can cut through the mesh in an arbitrary fashion.
  To cope with robustness problems caused by small cut elements, 
  we introduce ghost penalties in the vicinity of the embedded boundary
  to stabilize certain (semi)-norms associated with the advection and reaction operator.
  A few abstract assumptions on the ghost penalties are identified enabling us
  to derive geometrically robust and optimal a priori error 
  and condition number estimates for the stationary advection-reaction problem which hold
  irrespective of the particular cut configuration.
  Possible
  realizations of suitable ghost penalties are discussed. The
  theoretical results are corroborated by a number of computational 
  studies for various approximation orders and for two and
  three-dimensional test problems.
\end{abstract}


\section{Introduction}

To ease the burden of mesh generation in finite-element based
simulation pipelines, novel so-called \emph{unfitted}
finite element methods have gained much attention in recent years,
see~\cite{BordasBurmanLarsonEtAl2018} for a recent overview.  In
unfitted finite element methods, the mesh is exonerated from the task
to represent the domain geometry accurately and only used to define
proper approximation spaces.  The geometry of the model domain is
described independently by the means of a separate geometry model,
e.g., a CAD or level set based description, or even another,
independently generated tessellation of the geometry boundary.
Combined with suitable techniques to properly impose boundary or
interface conditions, for instance by means of Lagrange multipliers or
Nitsche-type
methods~\cite{GerstenbergerWall2008,CourtFournieLozinski2014,BurmanHansbo2010,BurmanHansbo2012,RankRuessKollmannsbergerEtAl2012,SchillingerRuess2014,
BoffiCavalliniGastaldi2015},
complex geometries can be simply embedded into an easy-to-generate background mesh.
Alternative routes to impose boundary and interface conditions for embedded geometries
are taken in the immersed finite element method~\cite{LiLinWu2003,LiIto2006} and in
finite element based formulations of the immersed boundary method~\cite{BoffiGastaldi2003}.

As the embedded geometry can cut arbitrarily through the background mesh,
a main challenge is to devise unfitted finite element methods that
are \emph{geometrically robust} in the sense that a priori error and conditioning
number estimates hold with constants
independent of the particular cut configuration.
One possible solution to achieve geometrical robustness
is provided by the approach taken in the 
\emph{cut finite element method} (CutFEM)
\cite{BurmanHansboLarsonEtAl2014}
which is based on a theoretically
founded stabilization framework.
The CutFEM stabilization technique allows to
transfers stability and approximation properties from a
finite element scheme posed on a standard mesh to its cut finite
element counterpart.
As a result, a wide range of problem classes
has been treated 
including, e.g.,
elliptic interface problems~\cite{BurmanZunino2012,GuzmanSanchezSarkis2015,BurmanGuzmanSanchezEtAl2017},
Stokes and Navier-Stokes type problems
\cite{BurmanHansbo2013,MassingLarsonLoggEtAl2013,BurmanClausMassing2015,
CattaneoFormaggiaIoriEtAl2014,MassingSchottWall2017,WinterSchottMassingEtAl2017,KirchhartGrosReusken2016,GrosLudescherOlshanskiiEtAl2016,GuzmanOlshanskii2016},
two-phase and fluid-structure
interaction
problems~\cite{SchottRasthoferGravemeierEtAl2015,GrosReicheltReusken2006,GrossReusken2011,MassingLarsonLoggEtAl2015}.
As a natural application area, unfitted finite element methods have
also been proposed for problems in fractured porous
media~\cite{FlemischFumagalliScotti2016,Fumagalli2012,
  DAngeloScotti2012,FormaggiaFumagalliScottiEtAl2013}.

In addition to the aforementioned unfitted \emph{continuous} finite
element methods, unfitted \emph{discontinuous Galerkin} (DG) methods have
successfully been devised to treat boundary and interface problems on
complex and evolving
domains~\cite{BastianEngwer2009,BastianEngwerFahlkeEtAl2011,Saye2015},
including flow problems with moving boundaries and
interfaces~\cite{SollieBokhoveVegt2011,HeimannEngwerIppischEtAl2013,Saye2017,MuellerKraemer-EisKummerEtAl2016,KrauseKummer2017,Kummer2017}.
In contrast to stabilized continuous cut finite element methods, in
unfitted discontinuous Galerkin methods, troublesome small cut
elements can be merged with neighbor elements with a large
intersection support by simply extending the local shape functions
from the large element to the small cut element.  As inter-element
continuity is enforced only weakly, no additional measures need to be
taken to force the modified basis functions to be globally continuous.
Consequently, cell merging in unfitted discontinuous Galerkin methods
provides an alternative stabilization mechanism to ensure that the
discrete systems are well-posed and well-conditioned.  For a very
recent extension of the cell merging approach to continuous finite
elements, we refer to~\cite{BadiaVerdugoMartin2018}.
Thanks to their favorable conservation and stability properties,
unfitted discontinuous Galerkin methods remain an attractive
alternative to continuous CutFEMs, but some drawbacks are the almost
complete absence of numerical analysis except
for~\cite{Massjung2012,JohanssonLarson2013},
the required rather invasive implementational changes
in standard DG assembly and solving routines to account for cell merging,
and the lack of natural discretization approaches for PDEs
defined on surfaces.

Starting with the early
contributions~\cite{ClarkeHassanSalas1986,Quirk1994}, cell merging
has also been a popular approach 
to mitigate small cut cell related stability issues and
severe time-step restrictions in explicit time-stepping methods
for finite volume based Cartesian cut
cell discretizations of compressible flow
problems~\cite{HartmannMeinkeSchroeder2008,SchneidersHartmannMeinkeEtAl2013},
and hyperbolic conservation laws in general.
For cut cell methods based on balancing numerical fluxes,
a major challenge is the robustness of the cell merging process, in particular
for complex 3D geometries or for problems with moving boundaries~\cite{Berger2017}.
Alternative approaches exist such as the $h$-box 
method~\cite{HelzelBergerLeVeque2005,BergerHelzel2012},
and stabilized fluxes in dimensionally split
methods~\cite{KleinBatesNikiforakis2009,GokhaleNikiforakisKlein2018},
but the resulting formulations 
are typically of at most second order.
For an overview on Cartesian cut cell methods
and approaches to overcome small cut cell related problems,
we refer to the recent review~\cite{Berger2017}.

So far, most of the theoretically analyzed unfitted finite element
schemes have been proposed for elliptic type of problems.
In contrast to conducted research on
finite volume based Cartesian cut cell methods,
very little attention has been paid to the theoretical development of unfitted
finite element methods for advection-dominant or hyberbolic
problems, except
for~\cite{MassingSchottWall2017,WinterSchottMassingEtAl2017}
considering a continuous interior penalty (CIP) based CutFEM for
the Oseen problem,
and~\cite{BurmanHansboLarsonEtAl2015b,OlshanskiiReuskenXu2014,BurmanHansboLarsonEtAl2018}
developing continuous unfitted finite element methods for
advection-diffusion equations on surfaces based on the CIP and
streamline upwind Petrov Galerkin (SUPG) approach. On the other
hand, for the DG community, hyperbolic problems have been of
interest from the very beginning, as the first DG method was
developed by~\cite{Reed1973} to model neutron transport problem.
The first analysis of these methods was presented
in~\cite{LesaintRaviart1974} and then improved
in~\cite{JohnsonNaevertPitkaeranta1984}, while
\cite{BrezziMariniSueli2004}
reformulated and generalized the upwind flux strategy in DG methods
for hyperbolic problems by introducing a tunable penalty parameter.
The advantageous conservation
and stability properties, the high locality and the
naturally inherited upwind flux term in the bilinear form makes DG
methods popular to handle specifically advection dominated
problems~\cite{Cockburn1999, HoustonSchwabSueli2002,HoustonSchwabSueli2002,Zarin2005} as
well as elliptic ones~\cite{ArnoldBrezziCockburnEtAl2002,
BrezziManziniMariniPietraRusso2000}. 
For a detailed overview,
we refer to the
monographs~\cite{DiPietroErn2012,HesthavenWarburton2007} which
include comprehensive bibliographies on discontinuous Galerkin methods
for both elliptic and hyberbolic problems.

\subsection{Novel contributions and outline of this paper}
\label{ssec:novel-contrib}
In this work we continue the development of a novel \emph{stabilized}
cut discontinuous Galerkin (CutDG) framework initiated
in~\cite{GuerkanMassing2018}. 
Departing from the elliptic boundary and
interface problems considered in~\cite{GuerkanMassing2018}, we here
propose and analyze CutDG methods for scalar, first-order hyperbolic
problems. 
The main focus is profoundly on stationary problems, but we shall briefly
demonstrate that the same type of method also extends to time-dependent problems.
However, the analysis of the time-dependent case is postponed to future
work. 
The proposed methods do not rely on the cell merging methodology, instead
geometrical robustness is achieved by adding properly
designed ghost penalty stabilization in the vicintiy of the embedded
boundary.  The idea of extending ghost penalty stabilization
techniques from the continuous CutFEM
approach~\cite{BurmanClausHansboEtAl2014} to discontinuous Galerkin
based discretizations allows for a minimally invasive extension of
existing fitted discontinuous Galerkin software to handle unfitted
geometries.  Only additional quadrature routines need to be added to
handle the numerical integration on cut geometries, and we refer to
the numerous quadrature algorithms capable of higher order geometry
approximation~\cite{MuellerKummerOberlack2013,Saye2015,Lehrenfeld2016,FriesOmerovic2015,FriesOmerovicSchoellhammerEtAl2017}
which have been proposed in recent years.  With a suitable choice of
the ghost penalty, the sparsity pattern of the matrix associated with the 
advection-reaction operator requires only little manipulation compared to its 
fitted DG counterpart.
In fact, the resulting stencil on elements near the embedded
boundary will correspond to the stencil of a symmetric interior
penalty method~\cite{Arnold1982} for second order diffusion
problems which in many DG-based finite element libraries is already
readily available. But similar to the cell merging approach, our
stabilization framework works automatically for higher order
approximation spaces with polynomial orders $p$ and is not limited
to low order schemes.

We start by briefly recalling the advection-reaction model problems
and the corresponding weak formulations in
Section~\ref{sec:model-problem}, followed by the presentation of
the stabilized cut discontinuous Galerkin methods in
Section~\ref{sec:cutDGM} which includes additional abstract
ghost penalties. A main feature of the presented numerical analysis
is that we identify a number of abstract assumptions on the
ghost penalty to prove geometrically robust optimal a priori
error and condition number estimates which hold
independent of the particular cut configuration.
To prepare the a priori error analysis, Section~\ref{ssec:useful-inequalities}
collects a number of useful inequalities and explains the
construction of an unfitted but stable $L^2$ projection operator.
In Section~\ref{ssec:stab-prop}, we derive an
inf-sup condition in a ghost penalty enhanced scaled streamline-diffusion
type norm. A key observation is that the classical argument
in~\cite{BrezziMariniSueli2004} based on boundedness on orthogonal
subscales does not work as the local orthogonality properties of
the unfitted $L^2$ projection is perturbed due to the mismatch between
the physical domain and the active mesh where the discontinuous
ansatz functions are defined. 
In Section~\ref{ssec:aprior-analysis}, the results from the previous
sections are combined to establish an optimal a priori error estimate of
the form $\mcO(h^{p+\onehalf})$ for 
the proposed stationary CutDG method with constants independent of
the particular cut configuration.
We continue our abstract analysis in
Section~\ref{ssec::condition-number-est} and prove that the condition
number of the matrix related to the advection-reaction operator 
scales like $\mcO(h^{-1})$, again with geometrically robust constants.
To the best of our knowledge, this is the first time, 
a theoretical analysis of any type of unfitted discontinuous Galerkin method
for an advection-reaction problem is presented.
Our theoretical investigation concludes with discussing a number of
ghost penalty realizations in Section~\ref{ssec:ghost-penalty-real}.
In Section~\ref{sec:time-stepping}, 
we briefly demonstrate how the ghost penalty approach can be
combined with explicit Runge-Kutta methods to solve
time-dependent advection-reaction problem under a standard
hyperbolic CFL condition.
Finally, we corroborate our theoretical analysis with numerical examples in
Section~\ref{ssec:numerical-results} where we study both the
convergence properties and geometrical robustness of the proposed CutDG
method.

\subsection{Basic notation}
Throughout this work, $\Omega \subset \RR^d$, $d = 2,3$ denotes an
open and bounded domain
with piecewise smooth boundary
$\partial \Omega$, while $\Gamma = \partial \Omega$ denotes its topological boundary.
For $U \in \{\Omega, \Gamma \}$
and $ 0 \leqslant m < \infty$, $1 \leqslant q \leqslant \infty$, let
$W^{m,q}(U)$ be the standard Sobolev spaces consisting of those
$\RR$-valued functions defined on $U$ which possess $L^q$-integrable
weak derivatives up to order~$m$. Their associated norms are denoted
by $\|\cdot \|_{m,q,U}$.  As usual, we write $H^m(U) = W^{m,2}(U)$ and
$(\cdot,\cdot)_{m,U}$ and $\|\cdot\|_{m,U}$ for the associated inner
product and norm. If unmistakable, we occasionally write
$(\cdot,\cdot)_{U}$ and $\|\cdot \|_{U}$ for the inner products and
norms associated with $L^2(U)$, with $U$ being a measurable subset of
$\RR^d$.
Any norm $\|\cdot\|_{\mcP_h}$ used in this work which
involves a collection of geometric entities $\mcP_h$ should be
understood as broken norm defined by
$\|\cdot\|_{\mcP_h}^2 = \sum_{P\in\mcP_h} \|\cdot\|_P^2$ whenever
$\|\cdot\|_P$ is well-defined, with a similar convention for scalar
products $(\cdot,\cdot)_{\mcP_h}$.  Any set operation
involving $\mcP_h$ is also understood as element-wise operation,
e.g., $ \mcP_h \cap U = \{ P \cap U \st P \in \mcP_h \} $ and $
\partial \mcP_h = \{ \partial P \st P \in \mcP_h \} $ allowing for
compact short-hand notation such as
$ (v,w)_{\mcP_h \cap U} = \sum_{P\in\mcP_h} (v,w)_{P \cap U} $ and
$ \|\cdot\|_{\partial \mcP_h\cap U} = \sqrt{\sum_{P\in\mcP_h} \|\cdot\|_{\partial P\cap
    U}^2}$.  Finally, throughout this work, we use the notation
$a \lesssim b$ for $a\leqslant C b$ for some generic constant $C$
(even for $C=1$) which varies with the context but is always
independent of the mesh size $h$ and the position of $\Gamma$ relative
to the background $\mcT_h$, but may depend on the dimensions
$d$, the polynomial degree of the finite element functions, the shape
regularity of the mesh and the curvature of $\Gamma$.

\section{Model problems}
\label{sec:model-problem}
For a given vector
field $b \in [W^{1,\infty}(\Omega)]^d$ and a scalar function $c \in
L^{\infty}(\Omega)$ we consider the advection-reaction problem of the form
\begin{subequations}
  \label{eq:adr-problem}
  \begin{alignat}{3}
    b \cdot \nabla u + c u &= f &&\quad \text{in } \Omega,
    \label{eq:adr-pde}
    \\
    u &= g && \quad \text{on } \Gamma^-,
    \label{eq:adr-boundary}
  \end{alignat}
\end{subequations}
\noindent where $g \in L^2(\Gamma^-)$ describes the boundary value on the inflow
boundary of $\Omega$ defined by
\begin{align}
  \Gamma^- &= \{ x \in \partial \Omega \st b(x)\cdot n_{\Gamma}(x) < 0\}.
  \label{eq:inflow-bdry}
             \intertext{Here, $n_\Gamma$ denotes the outer normal associated with $\Gamma$.
             Correspondingly, the outflow and characteristic boundary are defined by,
             respectively,}
  \Gamma^+ &= \{ x \in \partial \Omega \st b(x)\cdot n_\Gamma(x) > 0\},
  \label{eq:outflow-bdry}
  \\
  \Gamma^0\; &= \{ x \in \partial \Omega \st b(x)\cdot n_\Gamma(x) = 0\}.
  \label{eq:charact-bdry}
\end{align}
As usual in this setting,
we assume that
\begin{equation}
  c_0  \coloneqq \essinf_{x\in\Omega} \Bigl( c(x) - \frac{1}{2} \nDiv b(x) \Bigr) > 0.
  \label{eq:c0-def}
\end{equation}
The corresponding weak formulation of scalar hyperbolic problem~\eqref{eq:adr-problem}
is to seek $u \in V \coloneqq \{ v \in L^2(\Omega) \st  \bnabla v \in L^2(\Omega) \}$
such that $\foralls v \in V$
\begin{align}
  a(u,v) = l(v),
  \label{eq:weak-form-cont-problem}
\end{align}
with the bilinear form $a(\cdot,\cdot)$ and linear form $l(\cdot)$ given by
  \begin{align}
    a(u, v) = (b\cdot \nabla u + cu, v)_{\Omega} - (b\cdot n_{\Gamma} u, v)_{\Gamma^-},
              \qquad
    l(v) = (f,v)_{\Omega} - (b\cdot n_{\Gamma} g, v)_{\Gamma^-}.
    \label{eq:form-def-cont-probl}
  \end{align}

As we briefly consider time-dependent problems, we shall also solve
\begin{subequations}
  \label{eq:time-dep-adr-problem}
  \begin{alignat}{3}
    \partial_t u+ b \cdot \nabla u + c u &= f &&\quad \text{in } \Omega,
    \label{eq:time-dep-adr-pde}
    \\
    u &= g && \quad \text{on } \Gamma^-,
    \label{eq:time-dep-adr-boundary}
    \\
    \evaluated{u}{t=0} &= u_0  &&\quad \text{in } \Omega.
    \label{eq:time-dep-adr-initial}
  \end{alignat}
\end{subequations}
The weak form corresponding to \eqref{eq:time-dep-adr-problem} reads: find $u$ so that for each fix $t\in(0, T)$, $u\in V$ and such that $u$ fulfills
\begin{subequations}
\label{eq:weak-form-time-dep-cont-problem}
\begin{align}
  ( \partial_t u,v)_\Omega + a(u,v) &= l(v), \quad \foralls v \in V,\\
  \evaluated{u}{t=0} &= u_0.
\end{align}
\end{subequations}

\section{Cut discontinuous Galerkin methods for advection-reaction problems}
\label{sec:cutDGM}
Let $\widetilde{\mcT}_h$ be a background mesh covering $\overline{\Omega}$ 
consisting of $d$-dimensional, shape-regular (closed) polygons, $\{T\}$,
which are either simplices or quadrilaters/hexahedrals.
As usual, we introduce the
local mesh size $h_T = \diam(T)$ and the
global mesh size $h = \max_{T \in \widetilde{\mcT}_h} \{h_T\}$.
For $\widetilde{\mcT}_h$ we define the so-called \emph{active mesh}
\begin{align} 
  \mcT_h &= \{ T \in \widetilde{\mcT}_{h} \st T \cap \Omega^{\circ} \neq \emptyset \},
  \label{eq:active-mesh-bvp}
\end{align}
and its submesh $\mcT_{\Gamma}$ consisting of all elements cut by the boundary,
\begin{align} 
  \mcT_{\Gamma} &= \{ T \in\mcT_{h} \st T \cap \Gamma \neq \emptyset \}.
  \label{eq:cut-elements-mesh-bvp}
\end{align}
Note that since the elements $\{T\}$ are
closed by definition, the active mesh $\mcT_h$ still covers ${\Omega}$.
The set of interior faces in the active background mesh is given by
 \begin{align} 
   \mcF_h &= \{ F =  T^+ \cap T^- \st T^+, T^- \in \mcT_h \; \wedge \; T^+ \neq T^-\}.
  \label{eq:faces-interior-bvp}
 \end{align}
\begin{figure}[htb]
  \begin{center}
    \begin{minipage}[t]{0.40\textwidth}
    \vspace{0pt}
    \includegraphics[width=1.0\textwidth]{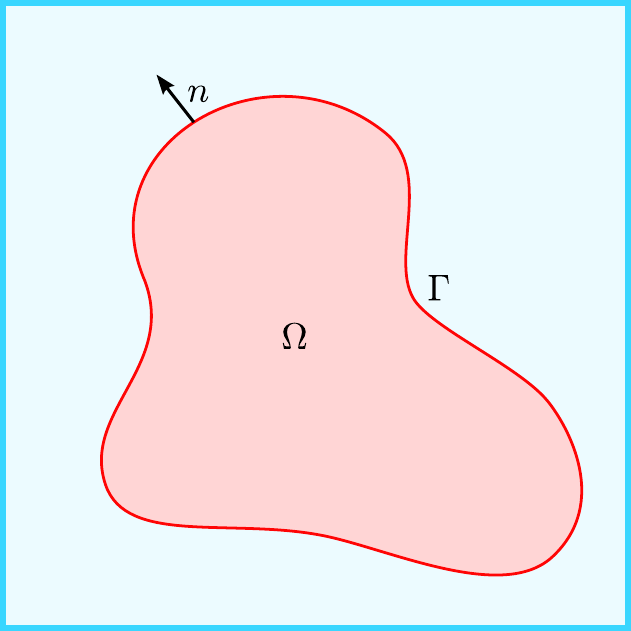}
  \end{minipage}
  \hspace{0.02\textwidth}
  \begin{minipage}[t]{0.40\textwidth}
    \vspace{0pt}
    \includegraphics[width=1.0\textwidth]{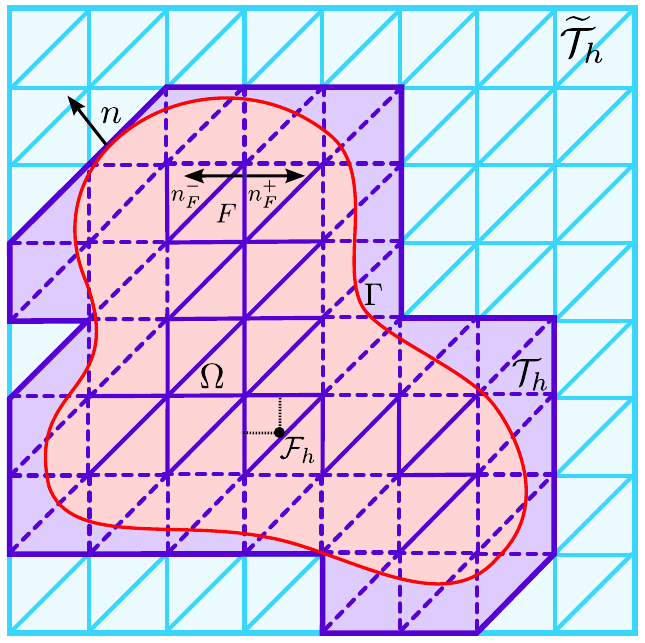}
  \end{minipage}
\end{center}
\caption{Computational domains for the problems~(\ref{eq:adr-problem}) and (\ref{eq:time-dep-adr-problem}).
  (Left) Physical domain $\Omega$ with boundary $\Gamma$ and outer normal $n$.
  (Right) Background mesh and active
  mesh used to define the approximation space.
  Dashed faces are used to define face based ghost penalties
  discussed in Section~\ref{ssec:ghost-penalty-real}.}
  \label{fig:domain-set-up}
\end{figure}
To keep the technical details the forthcoming numerical analysis at a moderate level,
we make two reasonable geometric assumptions on $\mcT_h$ and $\Gamma$.
\def\theassumption{G\arabic{assumption}}
\begin{assumption}
  The mesh $\mcT_h$ is quasi-uniform.
  \label{ass:mesh-quasi-uniformity}
\end{assumption}
\begin{assumption}
  \label{ass:fat-intersection-property}
  For $T\in \mcT_{\Gamma}$ there is an element $T'$ in $\omega(T)$
  with a ``fat'' intersection such that
  \begin{align} | T' \cap \Omega |_{d}
    \geqslant c_s |T'|_d
  \label{eq:fat-intersect-prop}
 \end{align}
for some mesh independent $c_s > 0$.
Here, $\omega(T)$ denotes the set of elements sharing at least one
node with $T$.
\end{assumption}
\begin{remark}
  Quasi-uniformity of $\mcT_h$ is assumed mostly for notational 
  convenience. Except for the condition number estimates, all derived estimates
  can be easily localized to element or patch-wise estimates.
\end{remark}
\begin{remark}
  The ``fat'' intersection property guarantees that
  $\Gamma$ is reasonably resolved by the active mesh $\mcT_h$.
  It is automatically 
  satisfied if, e.g.,  $\Gamma \in C^2$ and $h$ is small enough, that is; if
  $h \lesssim \min_{\{1,\ldots,d-1\}} \|\tfrac{1}{\kappa_i}\|_{L^{\infty}(\Gamma)}$
  where $\{\kappa_1, \ldots,\kappa_{d-1}\}$ are the  principal curvatures of $\Gamma$,
  see~\cite{BurmanGuzmanSanchezEtAl2017}.
\end{remark}
On the active mesh $\mcT_h$, we define the discrete function space $V_h$ as the broken polynomial space
of order~$p$,
\begin{align}
  V_h
  \coloneqq \PP_p(\mcT_h)
  \coloneqq \bigoplus_{T \in \mcT_h} \PP_p(T).
\label{eq:Vsh-def-bvp}
\end{align}

To formulate our cut discontinuous Galerkin discretization of the weak problem~(\ref{eq:weak-form-cont-problem}),
we recall the usual definitions of the averages
\begin{align}
\avg{\sigma}|_F &= \dfrac{1}{2} (\sigma_F^{+} + \sigma_F^{-}),
  \label{eq:mean-std-def-F}
  \\
  \avg{n_F \cdot \sigma }|_F &=
                             \dfrac{1}{2} n_F \cdot (\sigma_F^{+} + \sigma_F^{-}),
  \label{eq:mean-def-F}
\end{align}
and the jump across an interior facet $F \in
\mcF_h$, 
\begin{align}
  \jump{w}|_F &= w_F^{+} - w_F^{-}.
\end{align}
Here,  $w^\pm_F(x)= \lim_{t\rightarrow 0^+} w(x \mp t n_F)$ for some chosen unit facet normal $n_F$ on $F$.
\begin{remark}
  To keep the notation at a moderate level, we will from hereon drop any subscripts indicating
  whether a normal belongs to a face $F$ or to the boundary $\Gamma$ 
  as it will be clear from context.
\end{remark}
With these definitions in place,
our cut discontinuous Galerkin method
for advection-reaction problem~\eqref{eq:adr-problem} can be formulated as follows:
find $u_h \in V_h$ such that
  $\foralls v \in V_h$
\begin{align}
  A_h(u_h,v) \coloneqq a_h(u_h, v) + g_h(u_h, v) = l_h(v).
  \label{eq:cutDGM}
\end{align}
The discrete bilinear
forms $a_h(\cdot, \cdot)$ and $l_h(\cdot)$ represent the classical
discrete, discontinuous Galerkin counterparts
of~(\ref{eq:weak-form-cont-problem}) defined on mesh elements and
faces restricted to the physical domain $\Omega$.
More precisely, for $v, w \in V_h$, we let
\begin{align}
  a_h(v,w) 
  &= (b\cdot\nabla v +  cv,w)_{\mcT_h \cap \Omega} 
    - (b\cdot n v, w)_{\Gamma^-}
    \nonumber
  \\
  &\quad
  - (b\cdot n\jump{v},\avg{w})_{\mcF_h \cap \Omega}
  + \dfrac{1}{2} (|b\cdot n| \jump{v},\jump{w})_{\mcF_h \cap \Omega},
  \label{eq:ah-def}
  \\
  l_h(w) 
  &=
  (f,w)_{\Omega} - (b\cdot n g, w)_{\Gamma^-}.
  \label{eq:lh-def}
\end{align}
The main difference to the classical discontinuous Galerkin method
for Problem~\eqref{eq:adr-problem},
cf.~\cite{BrezziMariniSueli2004,DiPietroErn2012}, is the appearance of an
additional stabilization form $g_h$, also known as \emph{ghost penalty}.
The role of the ghost penalty is to ensure that the favorable
stability and approximation properties of the classical upwind stabilized
discontinuous Galerkin method carry over to the unfitted mesh scenario.
The precise requirements on $g_h$ will be formulated as a result of
the theoretical analysis performed in the remaining sections,
but we already point out that
since Problem~\eqref{eq:adr-problem}
involves an advection and a reaction term, it will be natural 
to assume that $g_h$  can be decomposed in a reaction related ghost penalty $g_c$
and an advection related ghost penalty $g_b$, both of which we assume to be symmetric,
\begin{align}
  g_h(v,w) = g_c(v,w) +  g_b(v,w) \quad \foralls v, w \in V_h.
  \label{eq:gb_gc_intro}
\end{align}

In the same manner, the CutDG method for the time-dependent problem~\eqref{eq:time-dep-adr-problem} reads, find $u_h$ such that for each fix $t\in(0, T)$, $u_h\in V_h$ and
\begin{subequations}
\label{eq:time-dep-cutDGM}
\begin{align}
  M_h \left(\partial_t u_h,v \right) + A_h(u_h,v) & = l_h(v), \quad \forall v\in V_h, \\
  \evaluated{u_h}{t=0} & = \pi_{M_h} u_0. \label{eq:initial-condition-cutDGM}
\end{align}
\end{subequations}
Here, $M_h$ denotes the stabilized $L^2(\Omega)$ inner product:
\begin{equation}
M_h(v,w)= (v,w)_\Omega + g_m(v,w), \quad v, w \in V_h,
\label{eq:stabilized_mass_form}
\end{equation}
where the only difference between $g_m$ and $g_c$ is a scaling by $c_0$:
\begin{equation}
g_c(v,w) = \tau_{c}^{-1} g_m(v,w), \quad v, w \in V_h.
\label{eq:mass_stabilization}
\end{equation}
In \eqref{eq:initial-condition-cutDGM}, $\pi_{M_h}: L^2(\Omega) \to V_h$ denotes the stabilized $L^2$ projection.
This projection is defined as the solution to the problem:
find $\pi_{M_h} u = \tilde{u}_h \in V_h$ such that
\begin{equation}
M_h(\tilde{u}_h,v)= (u,v)_\Omega, \quad \forall v \in V_h.
\label{eq:stabilized-L2-proj}
\end{equation}

  \begin{remark}
    For the sake of simplicity, we shall in the following sections restrict the 
    analysis to the case when the mesh consists of simplices.
    We also restrict the analysis to the method~\eqref{eq:cutDGM} for the stationary
    problem~\eqref{eq:adr-problem}, since this is our main focus.
    However, in Section~\ref{sec:time-dependent-convergence} we shall by
    numerical experiments verify that the method~\eqref{eq:time-dep-cutDGM} works
    for the time-dependent problem~\eqref{eq:time-dep-adr-problem} using quadrilaterals.
\end{remark}

\section{Norms, useful inequalities and approximation properties}
\label{ssec:useful-inequalities}
Following the presentation in~\cite{DiPietroErn2012},
we introduce a reference velocity $b_c$ and a reference time $\tau_c$,
\begin{align}
  b_c  = \| b \|_{0, \infty, \Omega},
  \qquad
 \tau_c^{-1} = \| c \|_{0,\infty,\Omega} + |b|_{1,\infty,\Omega},
  \label{eq:bc_tauc-def}
\end{align}
and define the central flux, upwind and the scaled streamline diffusion norm
by setting
\begin{align}
  \tn v \tncf^2 &= \tau_c^{-1}\|v\|_{\Omega}^2 +  \dfrac{1}{2} \| |b\cdot n|^{\onehalf} v \|_{\Gamma}^2,
  \label{eq:norm-cf-def}
  \\
  \tn v \tnup^2 &= \tn v \tncf^2 + \dfrac{1}{2}\| |b\cdot n|^{\onehalf} \jump{v}\|_{\mcF_h \cap \Omega}^2,
  \label{eq:norm-up-def}
  \\
  \tn v \tnsd^2 &= \tn v \tnup^2 + \| \phi_b^{\onehalf}  b \cdot \nabla v \|_{\mcT_h \cap \Omega}^2,
  \label{eq:norm-sd-def}
\end{align}
respectively. Thanks to the assumed quasi-uniformity, the global scaling factor $\phi_b$ is given by
\begin{align}
    \phi_b &= h/b_c.
  \label{eq:}
\end{align}
For each norm $\tn \cdot \tn_{t}$ with $t \in \{\mathrm{cf}, \mathrm{up}, \mathrm{sd}  \}$,
its ghost penalty enhanced version is given by
\begin{align}
  \tn v \tn_{t, h}^{2}  = 
  \tn v \tn_{t}^2 + |v|_{g_h}^2.
\end{align}
To simplify a number of estimates, we will also use
a slightly stronger norm than $\tn \cdot \tnsdh$, namely
\begin{align}
    \tn v \tnsdhast^2 &= \tn v \tncf^2 + \| \phi_b^{\onehalf}  b \cdot \nabla v \|_{\mcT_h \cap \Omega}^2
                       + b_c \| v \|_{\partial \mcT_h \cap \Omega}^2 + |v|_{g_h}^2.
    \label{eq:norm-sdhast-def}
\end{align}
Finally, we assume as in~\cite{DiPietroErn2012} that the
mesh~$\mcT_h$ is sufficiently resolved such that the inequality
\begin{align}
  h \leqslant b_c \tau_c
  \label{eq:mesh-resolution}
\end{align}
is satisfied, which means that Problem~\eqref{eq:adr-problem} can be considered 
as advection dominant on the element level, and that the advective velocity field 
$b$ is sufficiently resolved in the sense that the following estimates hold,
\begin{align}
  \| c \|_{0,\infty, \Omega} h \leqslant b_c,
  \qquad
  | b |_{1,\infty, \Omega} \leqslant \dfrac{b_c}{h}.
  \label{eq:advect-dominant-and-resolved}
\end{align}

Before we turn to the stability and a prior error analysis in
Section~\ref{ssec:stab-prop}, we
briefly review some useful inequalities needed later
and explain how a suitable approximation operator can be constructed on the active mesh $\mcT_h$.
Recall that for
$v \in H^1(\mcT_h)$, the local trace inequalities of the form
  \begin{align}
    \label{eq:trace-inequality}
    \|v\|_{\partial T}
    &\lesssim
    h^{-\onehalf} \|v\|_{T} +
    h^{\onehalf}  \|\nabla v\|_{T}
      \quad \foralls T\in \mcT_h,
    \\
    \|v\|_{\Gamma \cap T}
    &\lesssim
      h^{-\onehalf} \|v\|_{T}
      + h^{\onehalf} \|\nabla v\|_{T}
    \quad \foralls T \in \mcT_h,
    \label{eq:trace-inequality-cut}
  \end{align}
  hold, see~\cite{HansboHansboLarson2003} for a proof of the second one.
  For $v_h \in V_h$, let $D^j$ be the $j$-th total derivative, then the following
  inverse inequalities  hold,
\begin{alignat}{3}
  \| D^j v_h \|_{T \cap \Omega}
  &\lesssim 
    h ^{i-j}\| D^i v_h \|_{T}
   && \quad  \foralls T \in \mcT_h, &&\quad 0 \leqslant i \leqslant j,
  \label{eq:inverse-est-cut-T}
  \\
  \| D^j v_h \|_{\partial T \cap \Omega}
  &\lesssim
  h ^{i-j-\onehalf}\| D^i v_h \|_{T}
   && \quad  \foralls T \in \mcT_h, &&\quad 0 \leqslant i \leqslant j,
  \label{eq:inverse-est-cut-F}
  \\
  \| D^j v_h \|_{\Gamma \cap T}
  &\lesssim
  h ^{i-j-\onehalf}\| D^i v_h \|_{T}
   && \quad  \foralls T \in \mcT_h, &&\quad 0 \leqslant i \leqslant j,
  \label{eq:inverse-est-cut-Gamma}
\end{alignat}
see~\cite{HansboHansboLarson2003} for a proof of the last one.
Next, to define a suitable approximation operator, we depart from the $L^2$-orthogonal projection
  $\pi_h : L^2(\mcT_h) \to V_h$
  which for  $T \in \mcT_h$ and $F \in \mcF_T$ satisfies the error estimates
  \begin{alignat}{3}
    | v - \pi_h v |_{T,l}  &\lesssim h_T^{r-l} | v |_{r,T},
    && \quad
     0 \leqslant l \leqslant r,
    \\
    | v - \pi_h v |_{F,l}  &\lesssim h_T^{r-l-\onehalf} | v |_{r,T},
    && \quad 0  \leqslant l \leqslant r - 1/2,
  \end{alignat}
 where $r \coloneqq \min\{s, p+1\}$ and whenever $v \in H^s(T)$. Now to lift a function $v \in H^s(\Omega)$ to
 $H^s(\Omega_{h}^e)$, where we for the moment use the notation 
 $\Omega_h^e = \bigcup_{T \in \mcT_h} T$,
  we recall that for Sobolev spaces $W^{m,q}(\Omega)$, $0 < m \leqslant \infty, 1 \leqslant q \leqslant \infty$,
  there exists a bounded extension operator satisfying
  \begin{align}
    (\cdot)^e: W^{m,q}(\Omega) \to W^{m,q}(\RR^d), \quad \|v^e \|_{m,q, \RR^d} \lesssim \| v \|_{m,q,\Omega}
  \end{align}
  for $u \in W^{m,q}(\Omega)$, see~\cite{Stein1970} for a proof.
  We can now define an unfitted $L^2$ projection variant $\pi_h^e: H^s(\Omega) \to V_h$ by setting
  \begin{align}
   \pi_h^e v \coloneqq \pi_{h} v^e.
  \end{align}
  Note that this $L^2$-projection is slightly perturbed in the
  sense that it is not orthogonal on $L^2(\Omega)$ but rather on
  $L^2(\Omega_{h}^e)$. Combining the local approximation properties of $\pi_h$ with the stability
  of the extension operator $(\cdot)^e$, we see immediately that $\pi_{h}^e$
  satisfies the global error estimates
  \begin{alignat}{3}
    \| v - \pi_h^e v \|_{\mcT_h,l}  &\lesssim h^{r-l} \| v \|_{r,\Omega},
    && \quad  0 \leqslant l \leqslant r,
    \label{eq:interpol-est-cut-T}
   \\
    \| v - \pi_h^e v \|_{\mcF_h,l}  &\lesssim h^{r-l-\onehalf} \| v \|_{r,\Omega},
    && \quad  0 \leqslant l \leqslant r-1/2,
    \label{eq:interpol-est-cut-F}
    \\
    \| v - \pi_h^e v \|_{\Gamma,l}  &\lesssim h^{r-l-\onehalf} \| v \|_{r,\Omega},
    && \quad  0 \leqslant l \leqslant r-1/2.
    \label{eq:interpol-est-cut-Gamma}
  \end{alignat}
\begin{remark}
It would have been possible to use the projection $\pi_{M_h}$ from~\eqref{eq:stabilized-L2-proj},
instead of $\pi_h^e$, in the analysis of the method.
This is done in for example~\cite{BurmanClausMassing2015}.
However, using $\pi_h^e$ will simplify the following analysis slightly.
\end{remark}
\begin{remark}
In \eqref{eq:interpol-est-cut-T}--\eqref{eq:interpol-est-cut-Gamma}, 
we have intentionally neglected the superscript, $e$, on $v$, and assumed that 
use of the extension operator, $(\cdot)^e$, is implied.
In order to simplify the notation, we shall do so also in the following.
\end{remark}

To assure that the consistency error caused by $g_h$ does not affect the convergence order,
we require that the ghost penalties $g_b$ and $g_c$ are weakly consistent in the following sense.
\setcounter{assumption}{0}
\def\theassumption{HP\arabic{assumption}}
\begin{assumption}[Weak consistency estimate]
  \label{ass:weak-consistency}
  For $u \in H^s(\Omega)$ and $r = \min\{s, p+1\}$,
  the semi-norms $|\cdot|_{g_b}$ and $|\cdot|_{g_c}$ satisfy the estimates
  \begin{align}
    |\pi_h^e u|_{g_b} &\lesssim b_c^{\onehalf} h^{r-\onehalf} \|u\|_{r,\Omega},
    \\
    |\pi_h^e u|_{g_c} &\lesssim \tau_c^{-\onehalf} h^{r} \|u\|_{r,\Omega}.
  \end{align}
\end{assumption}

With these assumptions, the approximation error of the unfitted $L^2$ projection
can be quantified with respect to the $\tn \cdot \tnsdhast$ norm.
\begin{proposition}
Let $u \in H^{s}(\Omega)$ and assume that $V_h = \PP_p(\mcT_h)$ with $p \geqslant 0$. Then
for $r = \min\{s, p+1\}$, the approximation error of $\pi_h^e$ satisfies
  \begin{align}
    \tn u - \pi_h^e u \tnsdhast \lesssim
   (\tau_c^{-\onehalf}h^{\onehalf} 
    +b_c^{\onehalf})h^{r-\onehalf}\| u \|_{r, \Omega}.
    \label{eq:projection-error}
  \end{align}
\end{proposition}
\begin{proof}
  Set $e_{\pi} = u - \pi_h^e u$, then by definition
  \begin{align}
    \tn e_{\pi} \tnsdhast^2
    &= \tau_c^{-1} \| e_{\pi} \|_{\Omega}^2
    + \donehalf \||b\cdot n|^{\onehalf} e_{\pi}\|_{\Gamma}^2
    + b_c \| e_{\pi} \|_{\partial \mcT_h \cap \Omega}^2
      + \| \phi_b^{\onehalf} \bnabla e_{\pi} \|_{\Omega}^2
      + | e_{\pi} |^2_{g_h}
    \\
    &= I + II + III + IV +  | e_{\pi} |^2_{g_h}.
  \end{align}
  Thanks to Assumption~\ref{ass:weak-consistency}, we only need to estimate $I$--$IV$.
  Using~\eqref{eq:interpol-est-cut-T}, $I$ is bounded by
  \begin{align}
    I  &\lesssim \tau_c^{-1} h^{2r} \|u\|_{r,\Omega}^2,
  \end{align}
  while the last three terms can be bounded by combining~\eqref{eq:interpol-est-cut-T}--\eqref{eq:interpol-est-cut-Gamma} leading to
  \begin{align}
    II + III + IV
    &\lesssim
      b_c  \|h^{-\onehalf} e_{\pi}\|_{\mcT_h}^2
      + b_c  \|h^{\onehalf} \nabla e_{\pi}\|_{\mcT_h}^2
      \lesssim b_c  h^{2r - 1} \| u \|_{r, \Omega}^2.
  \end{align}
\end{proof}

\section{Stability analysis}
\label{ssec:stab-prop}
The goal of this section is to show that the proposed stabilized CutDG
method~(\ref{eq:cutDGM}) satisfies a discrete inf-sup condition
with respect to the $\tn \cdot \tnsdh$ norm, similar to the results
presented in~\cite{ErnGuermond2006a,DiPietroErn2012} for the corresponding classical
fitted discontinuous Galerkin method. In the unfitted case, one main challenge is
that the proof of the inf-sup condition involves various inverse estimates which in turn
forces us to gain control over the considered $\tn \cdot \tnsd$ norm on the
entire active mesh.

We start our theoretical analysis 
by recalling that a
partial integration of the element contributions in bilinear
form~\eqref{eq:ah-def} leads to a mathematically
equivalent expression for $a_h$, namely,
\begin{align}
  a_h(v,w) 
  &= ((c - \nabla \cdot b) v,w)_{\mcT_h \cap \Omega} 
  - (v, b\cdot\nabla w)_{\mcT_h \cap \Omega} 
  + (b\cdot n v, w)_{\Gamma^+}
  \nonumber
  \\
  &\quad
  + (b\cdot n\avg{v},\jump{w})_{\mcF_h \cap \Omega}
  + \dfrac{1}{2} (|b\cdot n| \jump{v},\jump{w})_{\mcF_h \cap \Omega}.
  \label{eq:ah-dual}
\end{align}
By taking the average of~\eqref{eq:ah-def} and~\eqref{eq:ah-dual},
we obtain the well-known decomposition of $a_h(\cdot, \cdot)$ into a
symmetric and skew-symmetric part,
\begin{align}
  a_h(v,w)  &= a_h^{\mathrm{sy}}(v,w)
              +  a_h^{\mathrm{sk}}(v,w),
              \label{eq:ah-avg}
              \intertext{where}
  a_h^{\mathrm{sy}}(v,w) &=
                           (c - \tfrac{1}{2}\nabla \cdot b) v,w)_{\mcT_h \cap \Omega}
  + \dfrac{1}{2} (|b\cdot n| v, w)_{\Gamma}
                             + \dfrac{1}{2} (|b\cdot n| \jump{v},\jump{w})_{\mcF_h \cap \Omega},
                             \label{eq:ah-symm}
                             \\
  a_h^{\mathrm{sk}}(v,w) &=
    \dfrac{1}{2}
    \bigl(
    (b\cdot\nabla v,w)_{\mcT_h \cap \Omega}
  - (v, b\cdot\nabla w)_{\mcT_h \cap \Omega} 
  \\
  &\quad
  -(b\cdot n\jump{v},\avg{w})_{\mcF_h \cap \Omega}
  + (b\cdot n\avg{v},\jump{w})_{\mcF_h \cap \Omega}
    \bigr).
                             \label{eq:ah-skew}
\end{align}
Consequently, the total bilinear form $A_h(\cdot, \cdot)$ is discretely coercive with respect to $\tn \cdot \tnuph$:
\begin{lemma}
  \label{lem:discrete-coerc}
  For $v \in V_h$ it holds
  \begin{align}
    c_0 \tau_c\tn v \tnuph^2 \lesssim A_h(v_h, v_h).
  \label{eq:discrete-coerc}
  \end{align}
\end{lemma}
\begin{proof}
  As in the classical fitted mesh case, the proof follows simply from
  observing that
  \begin{align}
    A_h(v,v) &= a_h(v,v) + g_h(v,v) = a_h^{\mathrm{sy}}(v,v) + g_h(v, v)
    \\
    &\geqslant 
    c_0 \| v \|_{\mcT_h \cap \Omega}^2
    + \dfrac{1}{2} \| |b\cdot n|^{\onehalf} v\|_{\Gamma}^2
    + \dfrac{1}{2} \| |b\cdot n|^{\onehalf} v\|_{\mcF_h \cap \Omega}^2
    + |v|_{g_h}^2
    \gtrsim c_0 \tau_c \tn v \tnuph^2,
  \end{align}
  noting that  $ 0 < c_0 \tau_c \lesssim 1$ thanks to assumption~\eqref{eq:c0-def}
  and the definition of $\tau_c$, cf.~\eqref{eq:bc_tauc-def}.
\end{proof}
\begin{remark}
  \label{rem:po0-stab}
  We point out that for the advection-reaction problem, the ghost
  penalty~$g_h$ is not needed to establish discrete coercivity in
  either the~\mbox{$\tn \cdot \tnup$} or ~\mbox{$\tn \cdot \tnuph$} norm, in contrast to the CutDG method for the Poisson
  problem presented in~\cite{GuerkanMassing2018}.
  Nevertheless, we will see that $g_h$ is required to prove an inf-sup
  condition for the discrete problem~\eqref{eq:cutDGM} using the
  stronger~\mbox{$\tn \cdot \tnsdh$} norm.
\end{remark}

Following the derivation of the inf-sup condition in the fitted mesh
case, see for instance~\cite{DiPietroErnGuermond2008}, a proper test
function $v \in V_h$ needs to be constructed which establishes
control over the streamline derivative. To construct such a suitable
test function, we assume the existence of discrete vector field
$b_h \in \PP_0(\mcT_h)$ satisfying
\begin{align}
  \| b_h^0 - b \|_{0, \infty, T} \lesssim h_T | b |_{1,\infty,T},
  \qquad
  \| b_h^0 \|_{0,\infty, T} \lesssim \| b^e\|_{0,\infty,T}.
  \label{eq:bh-prop}
\end{align}
Note that for $b \in [W^{1,\infty}(\Omega)]$ such a discrete vector field
$b_h^0$ can always be constructed. 
To establish the inf-sup result, the ghost penalties $g_b$ and $g_c$ for the discretized advection-reaction
problem are supposed to satisfy the following assumptions.  
\def\theassumption{HP\arabic{assumption}}
\begin{assumption}
  \label{ass:ghost-penalty-l2-ext}
  The ghost penalty $g_c$ extends the $L^2$ norm from the physical domain $\Omega$
  to the active mesh in the sense that
  for $v_h \in V_h$, it holds that
\begin{gather}
  \tau_c^{-1} \| v \|_{\mcT_h}^2
   \lesssim
  \tau_c^{-1} \| v \|_{\Omega}^2 + |v|_{g_c}^2
   \lesssim
  \tau_c^{-1} \| v \|_{\mcT_h}^2.
  \label{eq:ghost-penalty-tauc_l2norm-ext}
\end{gather}    
\end{assumption}
\begin{assumption}
  \label{ass:ghost-penalty-sdnorm-ext}
  The ghost penalty $g_b$ extends the streamline diffusion semi-norm from the physical domain $\Omega$
  to the active mesh in the sense that
  for $v_h \in V_h$, it holds that
\begin{gather}    
  \|\phi_b^{\onehalf} \bnabla v \|_{\mcT_h}^2 
  \lesssim
  \|\phi_b^{\onehalf} \bnabla v \|_{\Omega}^2 + |v|_{g_b}^2
  + \tau_c^{-1}\| v \|_{\Omega}^2 + |v|_{g_c}^2
  \lesssim
  \|\phi_b^{\onehalf} \bnabla v \|_{\mcT_h}^2 
  + \tau_c^{-1} \| v \|_{\mcT_h}^2.
  \label{eq:ghost-penalty-sdnorm-ext}
\end{gather}
\end{assumption}
\begin{remark}
  At first it might be more natural to assume that  
  \begin{align}
  \|\phi_b^{\onehalf} \bnabla v \|_{\mcT_h}^2  \sim
  \|\phi_b^{\onehalf} \bnabla v \|_{\Omega}^2 + |v|_{g_b}^2
  \end{align}
  is valid for $v \in V_h$ instead of the more convoluted
  estimate~\eqref{eq:ghost-penalty-sdnorm-ext}, but the forthcoming
  numerical analysis as well as the actual design of $g_b$ will
  require to frequently pass between certain locally constructed discrete vector fields $b_h$
  and the original vector field $b$. This is also the content of the
  next lemma.
\end{remark}
\begin{lemma}
  \label{lem:norm_rsb_beta_diff_const}
  Let $v \in V_h$ and let $P \subset \mcT_h$ be an element patch
  with $\diam(P) \sim h$.
  Let $b_h^P$ be a patch-wide defined velocity field satisfying
\begin{align}
  \| b_h^P - b \|_{0, \infty, P} \lesssim h | b^e |_{1,\infty,P},
  \qquad
  \| b_h^P \|_{0,\infty, P} \lesssim \| b^e\|_{0,\infty,P}.
  \label{eq:bhP-prop}
\end{align}
  Then
  \begin{align}
    {\phi}_{b}
    \|b_h^P - b\|_P^2
    \|\nabla v\|_{P}^2
    &\lesssim
    \tau_c^{-1}\| v \|_{P}^2.
      \label{eq:norm_rsb_beta_diff_const}
\end{align}
\end{lemma}
Before we turn to proof 
we note an immediate corollary  of Lemma~\ref{lem:norm_rsb_beta_diff_const}
and Assumption~\ref{ass:ghost-penalty-sdnorm-ext}.
\begin{corollary}
  \label{cor:ghost-penalty-sdnorm-ext-2}
  Let $\mcP_h$ be a collection of patches 
  and assume that the number of patch overlaps is uniformly bounded.
  Let $b_h^P$ and $P \in \mcP_h$ satisfy the assumptions in
  Lemma~\ref{lem:norm_rsb_beta_diff_const}.
  Then
\begin{align}
    \|\phi_b^{\onehalf} (b_h^P -b) \cdot \nabla v \|_{\mcP_h}^2  
    \lesssim
     \tn v \tnuph^2.
  \label{eq:ghost-penalty-sdnorm-ext-2}
\end{align}
  \label{cor:ghost-penalty-sdnorm-ext}
\end{corollary}
\begin{proof}[Proof of Lemma~\ref{lem:norm_rsb_beta_diff_const}]
Thanks to the assumptions~\eqref{eq:bhP-prop}
and the fact that $\phi_{b} \|b\|_{0,\infty,P} \lesssim h$,
we have the following chain of estimates,
\begin{align}
  \phi_{b} \|b_h - b\|_{0,\infty,P}^2 
  &\lesssim
  \phi_{b} 
  \|b\|_{0,\infty,P}
  h
  |b|_{1,\infty,P}
  \lesssim
  h^2
  |b|_{1,\infty,P}.
\end{align}
Thus an application of the Cauchy-Schwarz inequality
and the inequality~\eqref{eq:inverse-est-cut-T} to 
$\|\nabla v\|_{P}$ yields
\begin{align}
  \|
  {\phi}_{b}^{\onehalf}
  (b_h^P - b) \cdot \nabla v\|^2_{P}
  \lesssim
  \phi_{b}
  \|b_h^P - b\|_{0,\infty,P}^2 
  \| \nabla v\|^2_{P}
  \lesssim
  |b|_{1,\infty,\Omega} h^2 \|\nabla v \|_{P}^2
  \lesssim \tau_c^{-1} \|v\|_{P}^2.
\end{align}
\end{proof}
\begin{remark}
  Note that thanks to assumption $b \in W^{1,\infty}(\Omega)$ and the
  existence of an extension operator
  $(\cdot)^e : W^{1, \infty}(\Omega) \to W^{1,\infty}(\RR^d)$,
  cf.~Section~\ref{ssec:useful-inequalities}, a patch-wise defined
  velocity field $b_h^P$ satisfying
  Lemma~\ref{lem:norm_rsb_beta_diff_const}, Eq.~\eqref{eq:bhP-prop}
  can be always constructed, e.g., by simply taking the value of $b$
  at some point in the patch.
\end{remark}

Next, we state and prove the main result of this section, showing that
the cut discontinuous Galerkin method~(\ref{eq:cutDGM}) for the
advection-reaction problem~(\ref{eq:adr-problem}) is inf-sup stable
with respect to the \mbox{$\tn \cdot \tnsdh$} norm. The main challenge
here is to show that this result holds with a stability constant
which is independent of the particular cut configuration.
\begin{theorem}
  Let $A_h$ be the bilinear form defined by~\eqref{eq:cutDGM}.
  Then for $v \in V_h$ it holds that
  \begin{align}
   c_0 \tau_c \tn v \tnsdh \lesssim \sup_{w \in V_h\setminus\{0\}}\dfrac{A_h(v, w)}{\tn w \tnsdh}.
   \label{eq:inf-sup-condition}
  \end{align}
  with the hidden constant independent of the particular cut configuration.
\end{theorem}
\begin{proof}
  We follow the proof for the fitted DG method, see for instance~\cite{DiPietroErn2012},
  to construct a suitable test function $w \in V_h$
  for given $v \in V_h$ such that \eqref{eq:inf-sup-condition} holds.
  \\
  \noindent {\bf Step 1.}
  Setting $w_1 = v$, we gain control over the $\tn w \tnuph^2$, thanks to
  the coercivity result~(\ref{eq:discrete-coerc}),
  \begin{align}
    A_h(v, w_1) \geqslant c_0 \tau_c \tn v \tnuph^2.
    \label{eq:inf-sup-proof-step-1}
  \end{align}
  \noindent {\bf Step 2.}
  Next, set $w_2 = \phi_b b_h^0 \nabla v$ and note that $w_2 \in V_h$ since $b_h^0 \in \mcP_0(\mcT_h)$. Then
  \begin{align}
    A_h(v, w_2)
    &=
    \| \phi_b^{\onehalf} \bnabla v \|_{\Omega}^2
+ (\bnabla v, \phi_b  (b_h^0 - b)\cdot \nabla v)_{\Omega}
    + (c v, w_2)_{\Omega}
    \nonumber
    - (b\cdot n v, w_2)_{\Gamma^-}
    \\
    &\qquad
    - (b\cdot n \jump{v}, \mean{w_2})_{\mcF_h \cap \Omega}
    + \donehalf (|b\cdot n| \jump{v}, \jump{w_2})_{\mcF_h \cap \Omega}
    + g_h(v, w_2)
    \\
    &= \| \phi_b^{\onehalf} \bnabla v \|_{\Omega}^2
  + I + II + III + IV + V + VI.
  \label{eq:inf-sup-insert-vh2}
  \end{align}
  Term $I$ can be bounded by successively applying
  a Cauchy-Schwarz inequality and~\eqref{eq:norm_rsb_beta_diff_const},
  \begin{align}
    |I| \leqslant 
    \| \phi_b^{\onehalf} \bnabla v \|_{\Omega}
    \|\phi_b^{\onehalf} (b_h^0 - b) \cdot \nabla v\|_{\Omega}
    \lesssim 
    \| \phi_b^{\onehalf} \bnabla v \|_{\Omega}
    \tn v \tnuph.
    \label{eq:inf-sup-est-of-I}
  \end{align}
To estimate the remaining terms $II$--$VI$, let us for the moment assume that the stability estimate
\begin{align}
  \tn w_2 \tnsdhast
   \lesssim  
  \tn v \tnsdh
  \label{eq:uphast-norm-of-vh2}
\end{align}
holds. Then one can easily see that
\begin{align}
  |II| + \cdots + |VI|
  \lesssim
  \tn v \tnuph 
  \tn w_2 \tnsdhast
  \lesssim
  \tn v \tnuph 
  \tn v \tnsdh.
  \label{eq:inf-sup-est-of-II-VI}
\end{align}
Combing~\eqref{eq:inf-sup-est-of-I} and~\eqref{eq:inf-sup-est-of-II-VI},
we can estimate the right-hand side of~\eqref{eq:inf-sup-insert-vh2}
further by employing a Young inequality of the form 
$ab \leqslant \delta a^2 + \tfrac{1}{4 \delta} b^2$  yielding
\begin{align}
  A_h(v, w_2)
  &\geqslant
  \|\phi_b^{\onehalf} \bnabla v \|_{\Omega}^2
  -  \| \phi_b^{\onehalf} \bnabla v \|_{\Omega}  \tn v \tnuph 
  - C \tn v \tnuph \tn v \tnsdh
  \\
  &\geqslant
    \|\phi_b^{\onehalf} \bnabla v \|_{\Omega}^2
    - \delta \| \phi_b^{\onehalf} \bnabla v \|_{\Omega}
    -\dfrac{1}{4\delta} \tn v \tnuph 
  - \delta C \tn v \tnsdh^2
  - \dfrac{C}{4\delta} \tn v \tnuph^2
  \\
  &=
 (1- \delta - \delta C))\|\phi_b^{\onehalf} \bnabla v \|_{\Omega}^2
    - (\dfrac{1}{4\delta} + C\delta + \dfrac{C}{4\delta})  \tn v \tnuph^2
  \\
    &=
 \donehalf \|\phi_b^{\onehalf} \bnabla v \|_{\Omega}^2
    - \dfrac{(1+C)^3 + C }{2(1+C)} \tn v \tnuph^2,
    \label{eq:inf-sup-proof-step-2}
\end{align}
for some constant $C$ and  $\delta = \tfrac{1}{2+2C}$. This gives us the desired control
over the streamline derivative.
\\
{\bf Step 3.}
To construct a suitable test function for given $v$, 
we set now $w_3 = w_1 + \delta c_0 \tau_c w_2$. 
Thanks to stability estimate~\eqref{eq:uphast-norm-of-vh2} we have 
$
\tn w_3\tnsdh \lesssim \tn v \tnsdh
+ \delta c_0 \tau_c \tn v \tnsdh \leqslant (1+\delta)\tn v \tnsdh
$
and thus combining~\eqref{eq:inf-sup-proof-step-1} and~\eqref{eq:inf-sup-proof-step-2}
leads us to
\begin{align}
  A_h(v, w_3)
  &\geqslant
  (1- \delta \widetilde{C})c_0 \tau_c   \tn v \tnuph^2
  + \dfrac{\delta}{2}c_0 \tau_c \| \phi_b^{\onehalf} \bnabla v \|_{\Omega}^2
  \nonumber
  \\
  &\gtrsim c_0\tau_c \tn v \tnsdh^2
  \gtrsim c_0 \tau_c \tn v \tnsdh \tn w_3 \tnsdh
\end{align}
for some constant $\widetilde{C}$ and $\delta > 0$ small enough.
Dividing by $\tn w_3 \tnsdh$ and taking the supremum over $v$
proves~\eqref{eq:inf-sup-condition}.
To complete the proof, it remains to establish the stability bound~\eqref{eq:uphast-norm-of-vh2}.
\\
{\bf Estimate~\eqref{eq:uphast-norm-of-vh2}.}
We start by unwinding the definition of $\tn \cdot \tnsdhast$,
\begin{align}
\tn w_2 \tnsdhast^2
&= \tau_c^{-1} \| w_2 \|_{\Omega}^2
+ \|\phi_b^{\onehalf} \bnabla w_2\|_{\Omega}^2
+  \donehalf \| |b\cdot n|^{\onehalf}  w_2\|_{\Gamma}^2
+ b_c \|w_2\|_{\partial \mcT_h\cap\Omega}^2
+ |w_2|_{g_h}^2
\\
&=
I + II + III + IV + V.
\end{align}
The main tool to estimate $I$--$V$ is inequality~(\ref{eq:ghost-penalty-sdnorm-ext-2})
with $P = T$, $\mcP_h = \mcT_h$ and $b_h^P = b_h^0$.
Inserting the definition of $w_2 = \phi_b \bhnabla v$
and the fact  that $\phi_b \tau_c^{-1} \leqslant 1$
thanks to assumption~(\ref{eq:mesh-resolution}) yields
\begin{align}
  I 
  &= \tau_c^{-1} \| \phi_b \bhnabla v \|_{\Omega}^2
  = \tau_c^{-1} \phi_b \| \phi_b^{\onehalf} \bhnabla v \|_{\Omega}^2
  \leqslant \| \phi_b^{\onehalf} \bhnabla v \|_{\Omega}^2
    \lesssim
   \tn v \tnsdh^2.
\end{align}
   The second term $II$ can be dealt with by recalling
   the definition of $\phi_b$,
   applying the inverse estimate~(\ref{eq:inverse-est-cut-T})
   and subsequently moving to the streamline diffusion norm
   via~(\ref{eq:ghost-penalty-sdnorm-ext-2}),
   \begin{align}
   II &= \| \phi_b^{\onehalf}\bnabla (\phi_b \bhnabla v)  \|_{\Omega}^2
   \lesssim b_c h \| \nabla (\phi_b \bhnabla v)  \|_{\Omega}^2
   \lesssim b_c h^{-1} \| \phi_b \bhnabla v  \|_{\mcT_h}^2
   \\
   &\lesssim \| \phi_b^{\onehalf} \bhnabla v  \|_{\mcT_h}^2
   \lesssim \tn v \tnsdh^2.
\end{align}
    Next, invoking the inverse trace estimates~(\ref{eq:inverse-est-cut-Gamma})
    followed by an application of~(\ref{eq:ghost-penalty-sdnorm-ext-2}),
    we see that
\begin{gather}
  III
    \lesssim b_c\|\phi_b \bhnabla v\|^2_{\Gamma}
    \lesssim
    \|(h\phi_b)^{\onehalf} \bhnabla v\|^2_{\Gamma} 
    \lesssim
    \|\phi_b^{\onehalf} \bhnabla v\|^2_{\mcT_h}
    \lesssim
   \tn v \tnsdh^2,
      \intertext{and similarly for IV,}
      IV
    =   b_c\|\phi_b \bhnabla v\|^2_{\partial \mcT_h \cap \Omega}
    \lesssim  \|\phi_b^{\onehalf} \bhnabla v\|^2_{\mcT_h}
    \lesssim
   \tn v \tnsdh^2.
\end{gather}
After using~\eqref{eq:ghost-penalty-sdnorm-ext},
the remaining last term $V$ 
can be estimated by proceeding as for
 $I$ and $II$:
\begin{align}
  V
  &= |w_2|_{g_c}^2 + |w_2|_{g_b}^2
  \lesssim  \tau_c^{-1} \| w_2 \|_{\mcT_h}^2
  + \| \phi_b^{\onehalf} \bhnabla w_2 \|_{\mcT_h}^2
  \lesssim 
  \tn v \tnsdh^2.
\end{align}
\end{proof}
We conclude this section with a couple of remarks, elucidating the
role of the ghost penalties $g_c$ and $g_b$.
\begin{remark}
  We point out that the most critical part in the derivation of the
  inf-sup condition~(\ref{eq:inf-sup-condition}) is the proof of
  the stability estimate~\eqref{eq:uphast-norm-of-vh2}.  At several
  occasions it involves inverse estimates of the
  form~(\ref{eq:inverse-est-cut-T})--(\ref{eq:inverse-est-cut-Gamma})
  to pass from
  $w_2 = \phi_b \bhnabla v$ to $v$ in the streamline diffusion norm.
  As a result, we need to control the streamline derivative of $v$ on
  the \emph{entire} active mesh $\mcT_h$, which is precisely the role
  of the ghost penalty $g_b$.
\end{remark}
\begin{remark}
  The role of the ghost penalty $g_c$ twofold.
  In the previous proof, the ghost penalty
  $g_c$ is only needed to pass between $b$ and $b_h^0$ by controlling
  various norm expressions involving $b - b_h^0$ and $v$ in terms of
  $\tau_c^{-\onehalf} \|v\|_{\mcT_h}$. Consequently, the ghost penalty
  $g_c$ could have been omitted if $b \in
  \PP_1(\mcT_h)$ since then $\phi_b \bnabla v \in V_h$.
  Nevertheless, we will see that $g_c$ is indeed
  needed to ensure that the condition number of the system matrix is
  robust with respect to the boundary position relative to the
  background mesh, see Section~\ref{ssec::condition-number-est}.
\end{remark}

\begin{remark}
  The previous derivation rises the question whether the ghost penalty $g_b$
  is only needed because of the use of the stronger norm $\tn \cdot \tnsd$
  instead of the more classical $\tn \cdot \tnup$ norm. A closer look at the
  numerical analysis based on the $\tn \cdot \tnup$ norm as presented
  in~\cite{BrezziMariniSueli2004} reveals that $g_b$ cannot be avoided when
  $\PP_k$ elements with $k\geqslant 1$ are used. The reason is that the a
  priori analysis based on $\tn \cdot \tnup$ exploits the orthogonality of
  the $L^2$ projection operator $\pi_h$ to rewrite the advection related term
  in~\eqref{eq:ah-dual} to
  \begin{align}
    (\pi_h u - u, \bnabla v)_{\Omega}
    = (\pi_h u - u, (b - b_h^0) \cdot \nabla v)_{\Omega}
  \end{align}
  to arrive at an optimal error estimate for the advection related
  error terms in the $\tn u - u_h \tnup$ norm. In the unfitted case though, it is not
  possible to define a \emph{stable} orthogonal projection operator using the
  $L^2$ scalar product $(\cdot, \cdot)_{\Omega}$ considering only the
  physical domain $\Omega$.
  The orthogonality of the unfitted $L^2$ projection~$\pi_h^e$ constructed in Section~\ref{ssec:useful-inequalities}
  is perturbed as it
  satisfies only $(\pi_h^e u - u^e, v)_{\mcT_h} = 0$ for $v\in V_h$. Therefore, we get now
  \begin{align}
    (\pi_h^e u - u^e, \bnabla v)_{\Omega}
    = (\pi_h^e u - u^e, (b - b_h) \cdot \nabla v)_{\Omega}
    - (\pi_h u^e - u^e, b_h \cdot \nabla v)_{\mcT_h \setminus \Omega},
    \label{eq:perturbed-orthog-argument}
  \end{align}
  and thus we need to provide sufficient control of
  $ \| \phi_b^{\onehalf} b_h \cdot \nabla v\|_{\mcT_h \setminus \Omega}$
  in the relevant norms to handle the second term in~\eqref{eq:perturbed-orthog-argument} for $k \geqslant 1$.
  For $v \in \PP_0$ on the the other hand,
  \eqref{eq:perturbed-orthog-argument} vanishes, and thus no
  additional advection related CutDG stabilizations are needed.
    \label{rem:perturbed-orthog-argument}
\end{remark}

\section{A priori error analysis}
\label{ssec:aprior-analysis}
We turn to the a priori error analysis of the unfitted discretization
scheme~(\ref{eq:cutDGM}).  To keep the technical details at a moderate
level, we assume for a priori error analysis that the contributions
from the cut elements $\mcT_h \cap \Omega$, the cut faces
$\mcF_h \cap \Omega$ and the boundary parts $\Gamma \cap \mcT_h$ can
be computed exactly. For a thorough treatment of variational crimes
arising from the discretization of a curved boundary,
we refer the reader to
\cite{LiMelenkWohlmuthEtAl2010,BurmanHansboLarsonEtAl2014,GrossOlshanskiiReusken2014}.
We start with quantifying the effect of the stabilization $g_h$ on the
consistency of the total bilinear form $A_h$ defined
by~(\ref{eq:cutDGM}).
\begin{lemma}[Weak Galerkin orthogonality]
  \label{lem:weak-galerkin}
  Let $u \in H^1(\Omega)$ be the
  solution~\footnote{We assume
    $u \in H^1{(\Omega)}$ to simplify the presentation, but weaker
    regularity assumptions can be made, see for
    instance~\cite{DiPietroErn2012}.}
  to~(\ref{eq:adr-problem}) and let $u_h$ be the
  solution to the discrete formulation~(\ref{eq:cutDGM}).  Then
  \begin{align}
    a_h(u - u_h, v) - g_h(u_h, v) = 0
  \end{align}
\end{lemma}
\begin{proof}
  The proof is a direct consequence of the fact that $u$ satisfies $a_h(u,v) = l_h(v)\; \foralls v \in V_h$.
\end{proof}
\begin{theorem}[A priori error estimate for $p \geqslant 1$]
  \label{thm:apriori-est}
  For  $s \geqslant 2$, let $u \in H^{s}(\Omega)$ be the solution to
  the advection-reaction problem~(\ref{eq:adr-problem}) and let $u_h$ be
  the solution to the stabilized CutDG formulation~(\ref{eq:cutDGM}).
  Then with $r = \min\{s, p+1\}$ with $p \geqslant 1$ it holds that
  \begin{align}
  \label{eq:apriori-est}
    \tn u - u_h \tnsd &\lesssim (c_0 \tau_c)^{-1} b_c^{\onehalf} h^{r-\onehalf} \| u \|_{r,\Omega}.
  \end{align}
\end{theorem}
\begin{proof}
  Decompose $u - u_h$
  into a projection error $e_{\pi} = u - \pi_h u$ and
  a discrete error $e_h = \pi_h u - u_h$. Thanks to the interpolation estimate~(\ref{eq:projection-error}),
  it is enough to estimate the discrete error $e_h$ for which
  the inf-sup condition~(\ref{eq:inf-sup-condition})
  implies that
  \begin{align}
   c_0 \tau_c \tn e_h \tnsdh
    &\lesssim
      \sup_{v \in V_h\setminus{\{0\}}}
      \dfrac{A_h(e_h, v)}{\tn v \tnsdh}
      \\
    &=
      \sup_{v \in V_h\setminus{\{0\}}}
      \dfrac{a_h(e_{\pi}, v) + g_h(\pi_h u, v)}{\tn v \tnsdh}
      \label{eq:apriori-est-proof-step-2}
      \\
    &\lesssim
      \sup_{v \in V_h\setminus{\{0\}}}
      \dfrac{ \tn e_{\pi} \tnsdast \tn v \tnsd +  |\pi_h u|_{g_h}|v|_{g_h}}{\tn v \tnsdh}
      \\
    &\lesssim \tn e_{\pi} \tnsdast + |\pi_h u |_{g_h}.
  \end{align}
  Recalling assumption~\eqref{eq:mesh-resolution} and its
  reformulation~\eqref{eq:advect-dominant-and-resolved}, the bounds for
  the consistency and approximation error stated
  in Assumption~\ref{ass:weak-consistency} and
  Eq.\eqref{eq:projection-error} show that
\begin{align}
  \tn u - u_h \tnsd 
  &\lesssim 
  \tn e_{\pi} \tnsd
  + 
  \tn e_{h} \tnsdh
  \lesssim (c_0 \tau_c)^{-1} (\tn e_{\pi} \tnsdast + |\pi_h u |_{g_h})
  \\
  &\lesssim (c_0 \tau_c)^{-1} (\tau_c^{\onehalf}h^{\onehalf} + b_c^{\onehalf})
  h^{r-\onehalf} \| u \|_{r, \Omega}
  \lesssim (c_0 \tau_c)^{-1} b_c^{\onehalf}
  h^{r-\onehalf} \| u \|_{r, \Omega},
\end{align}
which concludes the proof
\end{proof}
We conclude this section by considering the particular case of
$\PP_0$ elements. First observe that since $\bnabla v = 0$ for $v
\in \PP_0(\mcT_h)$, all terms in \eqref{eq:perturbed-orthog-argument}
vanish. 
Referring to Remark~\ref{rem:perturbed-orthog-argument},
we therefore do not need the $\tn \cdot \tnsdh$ norm
to control the remainder term in 
\eqref{eq:perturbed-orthog-argument}.
Instead, we can work in the weaker and unstabilized $\tn \cdot
\tnup$ norm to derive an optimal a priori error
estimate.
\begin{proposition}[A priori error estimate for $p = 0$]
  \label{prop:apriori-est-p0}
  For  $s \geqslant 1$, let $u \in H^{s}(\Omega)$ be the solution to
  the advection-reaction problem~(\ref{eq:adr-problem}) and let $u_h$ be
  the solution to the \emph{unstabilized} CutDG formulation~(\ref{eq:cutDGM})
  with $g_h = 0$. Then,
  \begin{align}
  \label{eq:apriori-est-0}
    \tn u - u_h \tnup &\lesssim (c_0 \tau_c)^{-1} b_c^{\onehalf} h^{\onehalf} \| u \|_{1,\Omega}.
  \end{align}
\end{proposition}
\begin{proof}
  The proof follows verbatimly the standard arguments, see, e.g.,
  \cite{BrezziMariniSueli2004,DiPietroErn2012}. Starting from the
  discrete error $e_h = \pi_h u - u_h$, Remark~\ref{rem:po0-stab}
  together with the representation~\eqref{eq:ah-dual} and the fact
  that $\bnabla e_h = 0$ imply that
  \begin{align}
    c_0 \tau_c \tn e_h \tnup^2 
    &\lesssim a_h(e_h, e_h)
    \\
    &= a_h(e_{\pi}, e_h)
    \\
  &= ((c - \nabla \cdot b) e_{\pi},e_h)_{\mcT_h \cap \Omega} 
  + (b\cdot n e_{\pi}, e_h)_{\Gamma^+}
  \nonumber
  \\
  &\quad
  + (b\cdot n\avg{e_{\pi}},\jump{e_h})_{\mcF_h \cap \Omega}
  + \dfrac{1}{2} (|b\cdot n| \jump{e_{\pi}},\jump{e_h})_{\mcF_h \cap \Omega}.
  \\
  &\lesssim
  \tn e_{\pi} \tnupast \tn e_h \tnup,
  \end{align}
  where we defined the auxiliary norm 
  $\tn v \tnupast$
  by
  $\tn v \tnupast^2 = \tn v \tnup^2 + b_c \|v\|_{\partial \mcT_h \cap \Omega}^2$.
  Now the result follows from \eqref{eq:projection-error} since $\tn
  e_{\pi} \tnupast \leqslant \tn e_{\pi} \tnsdhast$.
\end{proof}
\begin{remark}
  The previous  proposition
  shows that geometrically robust optimal a priori error estimates
  can be obtained for unfitted $\PP_0$ based discretization
  of the \emph{stationary} advection-reaction problem,
  \emph{even without employing any ghost penalties}.
  Nevertheless, without adding $g_c$,
  the resulting system matrix will be severely ill-conditioned in the presence
  of small cut elements, see Section~\ref{ssec::condition-number-est}.
  Similar, an explicit time-stepping method for
  the time-dependent advection-reaction problem
  will suffer from severe time-step restrictions
  if the mass matrix is not sufficiently stabilized,
  which motivates the introduction of the
  stabilized $L^2$ scalar product in \eqref{eq:stabilized_mass_form}.
  \label{rem:p0-analysis}
\end{remark}

\section{Condition number estimates}
\label{ssec::condition-number-est}
We now conclude the numerical analysis of the CutDG method for the advection-reaction problem
by showing that the condition number associated with bilinear form (\ref{eq:cutDGM})
can be bounded by $C h^{-1}$ with a constant independent of particular cut configuration.
Our derivation is inspired by the presentation in~\cite{ErnGuermond2006}.

Let $\{\phi_i\}_{i=1}^N$ be the standard piecewise polynomial basis
functions associated with $V_h = \PP_p(\mcT_h)$ so that any $v \in V_h$
can be written as 
$v = \sum_{i=1}^N V_i \phi_i$ with coefficients $V = \{V_i\}_{i=1}^N \in \RR^N$.
The system matrix $\mcA$ associated with $A_h$ is defined by the relation
\begin{align}
  ( \mcA V, W )_{\RR^N}  = A_h(v, w) \quad \foralls v, w \in
  V_h.
  \label{eq:stiffness-matrix}
\end{align}
Thanks to the $L^2$ coercivity of $A_h$,
the system matrix $\mcA$ is a bijective linear mapping 
$\mcA:\RR^N \to \RR^N$ with its operator norm and condition number defined by
\begin{align}
  \| \mcA \|_{\RR^N}
  = \sup_{V \in \RR^N\setminus\{0\}}
  \dfrac{\| \mcA V \|_{\RR^N}}{\|V\|_{\RR^N}}
\quad \text{and}
\quad
  \kappa(\mcA) = \| \mcA \|_{\RR^N} \| \mcA^{-1} \|_{\RR^N},
  \label{eq:operator-norm-and-condition-number-def}
\end{align}
respectively.
To pass between the discrete $l^2$ norm of 
coefficient vectors $V$ and the continuous $L^2$ norm of
finite element functions $v_h$, we need to recall
the well-known estimate
\begin{align}
  h^{d/2} \| V \|_{\RR^N}
  \lesssim
   \| v \|_{L^2(\mcT_h)}
  \lesssim
   h^{d/2} \| V \|_{\RR^N},
  \label{eq:mass-matrix-scaling}
\end{align}
which holds for any quasi-uniform mesh $\mcT_h$ and $v\in V_h$.
Then we can prove the following theorem.
\begin{theorem}
  The condition number of the system matrix $\mcA$ associated with~(\ref{eq:cutDGM})
  satisfies
  \begin{align}
   \kappa(\mcA)  \lesssim b_c (c_0 h)^{-1}
  \end{align}
  independent of how the boundary $\Gamma$ cuts the background mesh~$\mcT_h$.
\end{theorem}
\begin{proof}
  The definition of the condition number requires to bound
  $\|\mcA\|_{\RR^N}$ and $\|\mcA^{-1}\|_{\RR^N}$.
  \\
  {\bf Estimate of $\|\mcA\|_{\RR^N}$.}  We start by estimating
  $A_h(v,w)= a_h(v,w) + g_h(v,w) \; \foralls v, w \in V_h$. To bound
  the first term $a_h$, successively employ a Cauchy-Schwarz
  inequality and the inverse trace
  inequalities~(\ref{eq:inverse-est-cut-F})
  and~(\ref{eq:inverse-est-cut-Gamma}) to obtain
  \begin{align}
    a_h(v,w)
    &\lesssim
    \|c\|_{0,\infty,\Omega} \|v\|_{\Omega} \|w\|_{\Omega}
    + \| \bnabla v \|_{\Omega} \|w\|_{\Omega}
    + \| |b\cdot n|^{\onehalf} v\|_{\Gamma}
      \| |b\cdot n|^{\onehalf} w\|_{\Gamma}
      \\
    &\qquad+ \| |b\cdot n|^{\onehalf} v\|_{\mcF_h \cap \Omega}
      \| |b\cdot n|^{\onehalf} w\|_{\mcF_h \cap \Omega}
    \\
    &\lesssim
    \|c\|_{0,\infty,\Omega} \|v\|_{\Omega} \|w\|_{\Omega}
      + h^{-1} b_c \|v\|_{\mcT_h} \|w\|_{\mcT_h}
    \lesssim
      h^{-1} b_c \|v\|_{\mcT_h} \|w\|_{\mcT_h},
  \end{align}
  where we used~(\ref{eq:advect-dominant-and-resolved}) in the last step.
  Next, note that by assumption~(\ref{eq:ghost-penalty-sdnorm-ext}), we have that
  \begin{align}
    |v|_{g_h}^2 \lesssim \|\bnabla v \|_{\mcT_h}^2 + \tau_c^{-1} \|v\|_{\mcT_h}^2
    \lesssim b_c h^{-1} \| v \|_{\mcT_h}^2,
  \end{align}
  and thus
  \begin{align}
      g_h(v,w)
    \lesssim  |v|_{g_h} |w|_{g_h}
    \lesssim b_c h^{-1} \| v \|_{\mcT_h} \|w\|_{\mcT_h}.
  \end{align}
  As a result of these estimates and~(\ref{eq:mass-matrix-scaling}), we have
\begin{align}
  A_h(v, w)
    \lesssim  b_c h^{-1}  \|v\|_{\mcT_h} \|w\|_{\mcT_h}
    \lesssim b_c h^{d-1}  \|V\|_{\RR^N} \|W\|_{\RR^N},
\end{align}
which allows us to estimate $\| \mcA \|_{\RR^N}$ by
\begin{align}
  \| \mcA \|_{\RR^N}
  &=
    \sup_{V \in \RR^N\setminus{\{0\}}}
    \sup_{W \in \RR^N\setminus{\{0\}}}
    \dfrac{(\mcA V, W)_{\RR^N}}{\|V\|_{\RR^N}\|W\|_{\RR^N}}
  \\
  &= \sup_{V \in \RR^N\setminus{\{0\}}}
    \sup_{W \in \RR^N\setminus{\{0\}}}
    \dfrac{A_h(v,w)}{\|V\|_{\RR^N}\|W\|_{\RR^N}}
    \lesssim
      b_c h^{d-1}.
      \label{eq:norm-A-matrix-est}
\end{align}
\\
{\bf Estimate of $\|\mcA^{-1}\|_{\RR^N}$.}
The discrete coercivity result~(\ref{eq:discrete-coerc})
and assumption~(\ref{eq:ghost-penalty-tauc_l2norm-ext}) implies that
\begin{align}
\label{eq:est-Ainv}
  A_h(v, v) \gtrsim
  c_0 \tau_c \tn v \tnuph
  \gtrsim c_0 \| v \|_{\mcT_h}^2
  \gtrsim c_0 h^{d} \| V \|_{\RR^N}^2,
\end{align}
and as a result,
  \begin{align}
  \|\mcA V  \|_{\RR^N}
  &= \sup_{W \in \RR^N\setminus{\{0\}}}
    \dfrac{(\mcA V, W)_{\RR^N}}{\|W\|_{\RR^N}}
    \geqslant
    \dfrac{(AV, V)_{\RR^N}}{\|V\|_{\RR^N}}
    =
    \dfrac{A_h(v,v)}{\|V\|_{\RR^N}}
    \gtrsim 
    c_0 h^d \|V\|_{\RR^N}.
  \end{align}
  Setting $V = \mcA^{-1} W$, the previous chain of estimates shows that
  $\| \mcA^{-1}\|_{\RR^n} \lesssim c_0^{-1} h^{-d}$ which in combination with~\eqref{eq:norm-A-matrix-est}
  gives the desired bound
  \begin{align}
    \| \mcA \|_{\RR^N} \| \mcA^{-1} \|_{\RR^N} \lesssim b_c (c_0h)^{-1}.
  \end{align}
\end{proof}

\section{Ghost penalty realizations}
\label{ssec:ghost-penalty-real}
In this section we present a number of possible realizations of the
ghost penalty operators satisfying our
Assumptions~\ref{ass:weak-consistency},
\ref{ass:ghost-penalty-l2-ext},
and~\ref{ass:ghost-penalty-sdnorm-ext}.  We briefly discuss the $L^2$
norm related ghost penalty $g_c$ first and then derive corresponding
realizations of $g_b$.  So far, three construction principles exist in
the literature.
The first one is the classical~\emph{face-based}
ghost penalty proposed by~\cite{BeckerBurmanHansbo2009} for continuous
$\PP_1$ elements.  High-order variants
were then introduced in~\cite{Burman2010}, for a detailed analysis we
refer to~\cite{Burman2010,MassingLarsonLoggEtAl2013a}.
In face-based ghost penalties, jumps of
normal-derivatives of all relevant polynomial orders across faces belonging to the face
set
\begin{align}
  \mcF_h^g
  = \{ F \in \mcF_h:  T^+ \cap \Gamma \neq \emptyset \lor T^- \cap \Gamma \neq \emptyset \}
  \label{eq:faces-gp-bvp}
\end{align}
are penalized. For our current problem class, the corresponding dG variant
extending the $L^2$ norm is given by
\begin{align}
  g_c^f(v,w) &= \gamma_c^f c_0 \sum_{j=0}^p h^{2j+1} (\jump{\partial_n^j v},  \jump{\partial_n^j w})_{\mcF_h^g},
  \label{eq:gc_face_based}
\end{align}
where the notation $\partialbar_n^j v \coloneqq \sum_{| \alpha | =
j}\tfrac{D^{\alpha} v(x)n^{\alpha}}{\alpha !}$ for multi-indices $\alpha
= (\alpha_1, \ldots, \alpha_d)$, $|\alpha| = \sum_{i} \alpha_i$ and
$n^{\alpha} = n_1^{\alpha_1} n_2^{\alpha_2} \cdots n_d^{\alpha_d}$ is
used. The constant $\gamma_c^f$ denotes a dimensionless stability parameter.

In \cite{Burman2010} and later in~\cite{BurmanHansbo2013},
alternative ghost penalties were proposed which were based on a
\emph{local projection stabilization}. For a given
patch $P$ of $\diam(P) \lesssim h$ containing the two elements $T_1$
and $T_2$, one defines the $L^2$ projection
$\pi_{P} : L^2(P) \to \PP_{p}(P)$ onto the space of polynomials of
order $p$ associated with the patch~$P$. For $v \in V_h$,
the fluctuation operator $\kappa_P = \Id - \pi_P$ measures then the deviation of
$v|_P$ from being a polynomial defined on $P$.
By choosing certain patch definitions, 
a coupling between
elements with a possible small cut and those with a fat
intersection is ensured. One patch choice arises naturally from the definition
of $\mcF_h^g$ by defining
the patch $P(F) = T^+_F \cup T^-_F$ for two elements $T^+_F$, $T^-_F$ sharing the interior face $F$ and
setting
\begin{align}
  \mcP_1 = \{P(F) \}_{F\in \mcF_h^g}.
  \label{eq:patches-face-based}
\end{align}
A second possibility is to use neighborhood patches $\omega(T)$,
\begin{align}
  \mcP_2 = \{\omega(T) \}_{T \in \mcT_{\Gamma}}.
\end{align}
Finally, one can mimic the cell agglomeration approach taken in
classical unfitted discontinuous Galerkin
approaches~\cite{BastianEngwer2009,HeimannEngwerIppischEtAl2013,SollieBokhoveVegt2011,JohanssonLarson2013}
by associating to each cut element $T \in \mcT_{\Gamma}$ with a small
intersection $|T \cap \Omega|_d \ll |T|_d$ an element
$T' \in \omega(T)$ satisfying the fat intersection property
$|T' \cap \Omega|_d \geqslant c_s |T'|^d$.
Introducing the ``agglomerated
patch'' $P_a(T) = T \cup T'$, a proper collection of patches
is given by
\begin{align}
  \mcP_3 = \{P_a(T)  \st T \in \mcT_{\Gamma} \wedge  |T \cap \Omega|_d \leqslant c_s |T|_d \}.
\end{align}
The resulting \emph{local projection} based ghost penalty extending the $L^2$ norm in the sense
of~\ref{ass:ghost-penalty-l2-ext} are then defined as follows:
\begin{align}
  g_c^p(v,w) &= \gamma_c^p c_0 \sum_{P \in \mcP} ( \kappa_P v,  \kappa_P w)_P, \quad \mcP \in \{\mcP_1, \mcP_2, \mcP_3\}.
\end{align}
Finally, an elegant version of a patch-based ghost penalty
avoiding  the assembly of local projection matrices completely 
was recently proposed in~\cite{PreusLehrenfeldLube2018}.
Using the natural global polynomial extension $u_i^e$
every polynomial $u_i$ on an element $T_i$ possesses,
a volume-based jump on a patch $P = T_1 \cup T_2$ can be defined by
$\jump{u}_P = u_1^e - u_2^e$ which give raise to the \emph{volume based} ghost penalty
\begin{align}
  g_c^v(v,w) &= \gamma_b^v c_0 \sum_{P \in \mcP_1} (\jump{v}_P,\jump{w}_P)_P, \quad \mcP \in \{\mcP_1, \mcP_3\}.
\end{align}
\begin{lemma}
    Each of the face, projection and volume-based realizations of $g_c$ 
    satisfies Assumptions~\ref{ass:weak-consistency} and \ref{ass:ghost-penalty-l2-ext}.
  \end{lemma}
  \begin{proof}
    As the original proofs require only minimal adaption to derive \ref{ass:weak-consistency} and
    \ref{ass:ghost-penalty-l2-ext} for discontinuous ansatz functions instead of continuous ones,
    we refer to~\cite{Burman2010} and~\cite{MassingLarsonLoggEtAl2013a}
    for the analysis of $g_h^p$ and $g_h^f$, and to~\cite{PreusLehrenfeldLube2018} for the volume penalty $g_h^v$.
  \end{proof}

  Next, we design proper realization of the advection related ghost penalty $g_b$.
  The natural idea is to start from $g_c$ and to replace $v$ with $\phi_b^{\onehalf} \bnabla v$
  as we wish to control the $L^2$ norm of the scaled streamline derivative.
  This idea leads us to the following lemma.
  \begin{lemma}
    \label{lem:ghost-penalty-strdiff}
    Each of the face, projection and volume-based realizations of $g_b$ defined by
  \begin{align}
    \label{eq:ghost-penalty-ghf}
    g_b^f(v,w) &= \gamma_b^f  \sum_{j=0}^p \phi_b h^{2j+1} (\jump{b_h^P \cdot \nabla \partial_n^j v},
                 \jump{b_h^P \cdot \nabla\partial_n^j w})_{\mcF_h^g},
               \\
    \label{eq:ghost-penalty-ghp}
  g_b^p(v,w) &= \gamma_b^p \sum_{P \in \mcP} \phi_b (\kappa_p  (b_h^P \cdot  \nabla v) , \kappa_p ( b_h^P \cdot  \nabla w))_P, \quad \mcP \in \{\mcP_1, \mcP_2, \mcP_3\},
               \\
    \label{eq:ghost-penalty-ghv}
    g_b^v(v,w) &= \gamma_b^v \sum_{P \in \mcP_1} \phi_b
                 (\jump{b_h^P \cdot \nabla v}_P, \jump{b_h^P \cdot \nabla w}_P)_P, \quad \mcP \in \{\mcP_1, \mcP_3\}
  \end{align}
  satisfies Assumption~\ref{ass:weak-consistency} and \ref{ass:ghost-penalty-sdnorm-ext}.
  Here, we choose $b_h^P$ to be a patch-wise defined, constant vector-valued function satisfying
  the assumptions of Lemma~\ref{lem:norm_rsb_beta_diff_const}.
  For $g_b^f$, the vector field $b_h^P$ is understood to be defined on each face patch $P(F) = T_F^+ \cup T_F^-$.
\end{lemma}
\begin{proof}
  Here, we consider only $g_b^v$ as the analysis for $g_b^p$ and
  $g_b^f$ is rather similar, and a detailed theoretical analysis of
  the $g_b^f$ can be also found in~\cite{MassingSchottWall2017}.
  
  We start with the verification of Assumption~\ref{ass:ghost-penalty-sdnorm-ext}.
  First recall, that for two incident elements $T_1$ and $T_2$ and
  $P = T_1 \cup T_2$, the norm equivalence
  $\| w \|_{T_1}^2 \sim \| w \|_{T_2}^2 + \| \jump{w}_P \|_{P}^2$
  holds for $w\in V_h$. Setting
  $w = \phi_b^{\onehalf} b_h^P \cdot \nabla v$ and using
  Lemma~\ref{lem:norm_rsb_beta_diff_const},
  Eq.(\ref{eq:norm_rsb_beta_diff_const}), to pass between
  $\phi_b^{\onehalf} b_h^P \cdot \nabla v$ and
  $\phi_b^{\onehalf} b_ \cdot \nabla v$, we conclude that
  \begin{align}
    \| \phi_b^{\onehalf} \bnabla v \|_{T_1}^2
    & \lesssim
      \| \phi_b^{\onehalf} b_h^P \cdot \nabla v \|_{T_1}^2
      + \| \phi_b^{\onehalf} (b_h^P - b) \cdot \nabla v\|_{T_1}^2
      \\
     &\lesssim
      \| \phi_b^{\onehalf} b_h^P \cdot \nabla v \|_{T_1}^2
      + \tau_c^{-1} \| v\|_{T_1}^2
      \\
    & \lesssim
      \| \phi_b^{\onehalf} b_h^P \cdot \nabla v \|_{T_2}^2
     + \| \jump{\phi_b^{\onehalf} b_h^P \cdot \nabla v}_P \|_{P}^2
      + \tau_c^{-1} \| v\|_{T_1}^2
    \\
    &\lesssim
    \| \phi_b^{\onehalf} \bnabla v \|_{T_2}^2
      + \tau_c^{-1} \| v\|_{T_2}^2
     + \| \jump{\phi_b^{\onehalf} b_h^P \cdot \nabla v}_P \|_{P}^2
      + \tau_c^{-1} \| v\|_{T_1}^2.
  \end{align}
  This together with the geometric assumption~\ref{ass:fat-intersection-property}
  and
  the norm equivalence
  $\| \phi_b^{\onehalf} \bnabla v \|_{T_2}^2 \sim \| \phi_b^{\onehalf} \bnabla v \|_{T_2\cap\Omega}^2$
  on elements $T_2$ with a fat intersection implies that
  \begin{align}
    \| \phi_b^{\onehalf} \bnabla v \|_{\mcT_h}^2
    \lesssim
    \| \phi_b^{\onehalf} \bnabla v \|_{\Omega}^2
    + |v|_{g_b^v}^2  + \tau_c^{-1} \| v \|_{\Omega} + |v|_{g_c}^2,
  \end{align}
  proving the first inequality in Assumption~\ref{ass:ghost-penalty-sdnorm-ext}.
  The second inequality easily follows from assumption~\ref{ass:ghost-penalty-sdnorm-ext}
  implying that  $|v|_{g_c}^2\lesssim \tau_c^{-1} \|v\|_{\mcT_h}^2$ and the fact
  that $|v|_{g_b^v}^2\lesssim
  \| \phi_b^{\onehalf} \bnabla v \|_{\mcP}^2 + \tau_c^{-1} \| v \|_{\mcP}^2$.

  To prove the weak consistency assumption~\ref{ass:weak-consistency},
  simply observe that
  $\jump{ \phi_b^{\onehalf} b_h^P \cdot \nabla \pi_p u^e}_P = 0$ for $u \in H^s(\Omega)$, $s \geqslant 1$,
  and the patch-wise defined $L^2$ projection $\pi_P$, 
  and thus
  \begin{align}
    \sum_{P \in \mcP}\| \jump{ \phi_b^{\onehalf} b_h^P \cdot \nabla \pi_h^e u } \|_{P}^2
    &=
    \sum_{P \in \mcP}\| \jump{ \phi_b^{\onehalf} b_h^P \cdot \nabla (\pi_h^e u - \pi_P u^e) } \|_{P}^2
    \\
    &\lesssim
    \sum_{P \in \mcP}\| \phi_b^{\onehalf} b_h^P \cdot \nabla (\pi_h^e u - u^e ) \|_{P}^2
    + \sum_{P \in \mcP}\| \phi_b^{\onehalf} b_h^P \cdot \nabla ( u - \pi_P u^e ) \|_{P}^2
    \\
    &\lesssim
      \phi_b b_c^2 h^{2(r-1)} \|u^e\|_{r,\mcT_h}^2
      \lesssim
     b_c h^{2r -1} \|u\|_{r,\Omega}^2.
  \end{align}
  which concludes the proof.
\end{proof}
\begin{remark}
  \label{rem:unified-gh}
  It is possible to replace the convection and reaction related
  stabilization forms $g_b$ and $g_c$ by a unified ghost
  penalty.  For instance, for the face-based
  realizations~(\ref{eq:gc_face_based})
  and~(\ref{eq:ghost-penalty-ghf}),
  a single ghost penalty
  of the form
  \begin{align}
  \label{eq:unified-ghf}
    g^f(v,w) &= \gamma^f \sum_{j=0}^p  \Bigl(c_0  +
               \frac{b_c}{h}\Bigr)  h^{2j+1} (\jump{\partial_n^j v},  \jump{\partial_n^j w})_{\mcF_h^g},
  \end{align}
  can be used instead, at the cost of introducing some
  additional (order-preserving) cross-wind diffusion.
  Consequently, only face jump penalties of the form $[\partial_n^j (\cdot)]$ need to be implemented.
  We refer to~\cite{MassingSchottWall2017}[Remark 3.4] for a detailed discussion.
\end{remark}

\section{Time-Stepping}
 \label{sec:time-stepping}
After the detailed theoretical analysis of the stabilized CutDG
method for the stationary advection-reaction problem in the
previous sections, we now formulate a CutDG method for the
time-dependent problem~\eqref{eq:time-dep-cutDGM}.
The presented formulation is only thought as 
a first demonstration of how ghost penalty techniques
can be employed to devise explicit time-stepping
methods which do not suffer from a severe time-step
restriction caused by small cut elements.
A detailed theoretical analysis is beyond the scope of the present
work and will be subject of a forthcoming paper.

Following and extending the notation in
Section~\ref{ssec:wavy-boundary-example}, the mass matrix~$\mcM$
associated with the stabilized $L^2$ inner product is defined by
the relation
\begin{equation}
  (\mcM V, W)_{\RR^N} = M_h(v, w) \quad \foralls v,w \in V_h,
  \label{eq:mass-matrix-def}
\end{equation}
where again $V, W \in \RR^N$ denote the coefficient vectors 
of $v$ and $w$. Denoting the time-depending coefficient vector
$U(t) \in$ associated with the semi-discrete solution $u_h(t)$,
we are left with solving a system of the form
\begin{subequations}
  \label{eq:ode-system-def}
\begin{align}
U'(t)  & =  \mcL U(t) + F(t),\\
U(0) & = U_0,
\end{align}
\end{subequations}
where we set $\mcL = \mcM^{-1} \mcA$
and defined  the vector $F(t) \in \RR^N$ 
by the relation
 $ (\mcM V, F(t))_{\RR^N} = (f(t), v_h)_{\Omega}$;
 that is, $F$ is usual load vector premultiplied with $\mcM^{-1}$.
To arrive at a fully discrete system, 
we shall use either the forward Euler method,
\begin{equation}
U^{n+1}=U^n + \Delta t \left ( \mcL U^n+F(t^n) \right ),
\label{eq:RungeKutta1}
\end{equation}
or the following explicit third order Runge-Kutta method taken from
\cite{DiPietroErn2012,BurmanErnFernandez2010}
\begin{subequations}
\label{eq:RungeKutta3}
\begin{align}
U^{n,1} & = U^n + \Delta t ( \mcL U^n +F(t^n) ),\\
U^{n,2} & = \tfrac{1}{2} (U^{n}+U^{n,1}) + \tfrac{\Delta t}{2} ( \mcL U^n  + F(t^n) + \Delta t F'(t^n) ),\\
U^{n + 1} & = \tfrac{1}{3}( U^{n} + U^{n,1} + U^{n,2} ) + \tfrac{\Delta t}{3} 
\bigl( 
  \mcL U^{n,2}  + F(t^n) + \Delta t F'(t^n) + \tfrac{\Delta t^2}{2} F''(t^n)  
\bigr).
\end{align}
\end{subequations}
In \eqref{eq:RungeKutta3}, $U^{n,s}$, $s\in \{1,2\}$ denotes the intermediate stages.
For simplicial meshes, the combination of these time-stepping 
methods with standard fitted DG-method that corresponds to 
\eqref{eq:time-dep-cutDGM} were studied in \cite{BurmanErnFernandez2010, DiPietroErn2012}. 
It was shown that under a hyperbolic CFL condition $h$: $\Delta t = C h$,
\eqref{eq:RungeKutta1} combined with $\PP_0(\mcT_h)$ converges
as $\mcO(h^{1/2} + \Delta t)$, and
\eqref{eq:RungeKutta3} combined with $\PP_2(\mcT_h)$ converges
as $\mcO(h^{5/2}+\Delta t^3)$.
For details, we refer to \cite[Section~3.1]{DiPietroErn2012} and
\cite{BurmanErnFernandez2010}.

\begin{remark}
  If the mass-matrix related ghost penalty $g_m$ uses 
  face-based stabilizations of the form~\eqref{eq:gc_face_based}
  together with the face set~\eqref{eq:faces-gp-bvp}, the mass matrix
  will no longer be block diagonal as in standard non-cut DG methods,
  since the elements in $\mcT_{\Gamma}$  are interlaced with each other through the
  face set~$\mcF_h^g$. As a remedy, one could
  employ local projection or volume based ghost penalties
  instead, together with a more careful choices of the associated
  patches $\mcP$, so that in the vicinity of the embedded boundary,
  the mass matrix can be inverted locally on patches.
\end{remark}

\section{Numerical results}
\label{ssec:numerical-results}
This section is devoted to a number of numerical experiments
supporting our theoretical findings. We first investigate the 
method~\eqref{eq:cutDGM} for the stationary problem. Convergence rate studies 
for smooth manufactured solutions defined on complex 2D and
3D domains are presented to demonstrate the applicability of our
method to non-trivial domain geometries. The goal of the second,
more detailed convergence study is two-fold. First, we examine the
experimental order of convergence when using higher order elements
up to degree $p=3$. Second, using the same domain set-up, we
consider a typical test case of a rough solution with a sharp
internal layer to illustrate that the proposed CutDG method
possesses the same robustness as the corresponding standard fitted
DG method. We then consider two numerical studies
illustrating the importance of the ghost penalties for the
geometrical robustness of the derived a priori error and condition
number estimates.
Finally, we conclude this section by a brief study of the time-dependent method 
\eqref{eq:time-dep-cutDGM}. As for the stationary problem, we examine both the 
order of convergence and the geometrical robustness of the method.

\subsection{Convergence rate experiments}
In the first series of experiments, we test the convergence of
the proposed stationary CutDG method \eqref{eq:cutDGM} over various geometries, dimensions, and
polynomial orders. Common to all examples below, a background mesh
$\widetilde{\mcT}_0$ for the embedding domain
$\widetilde{\Omega}=[-a,a]^d$ is generated and successively refined, from which
the active background meshes $\{\mathcal{T}_k\}^N_{k=0}$ are
extracted. The mesh size is then given by $h_k = 2a\cdot 2^{-3-k}$.
For a given polynomial order $p$ and
active mesh $\mathcal{T}_k$, we
compute the numerical solution $u_k^p \in \PP_p(\mcT_k)$
from~\eqref{eq:cutDGM}
using the ghost penalties
defined in \eqref{eq:ghost-penalty-ghf} and
\eqref{eq:gc_face_based}. The experimental order of convergence
(EOC) shown in the convergence tables below is calculated using
\begin{align}
  \text{EOC}(k,p) =
  \dfrac{\log(E_{k-1}^p/E_{k}^p)}{\log(h_{k-1}/h_k)}, 
\end{align} 
where $E_k^p
= \| e_k^p \| = \| u - u_k^p \| $ denotes the  error of the numerical
approximation $u_k^p$ measured in a certain (semi-)norm $\|\cdot
\|$.    
The error norms considered in our tests are the $L^2$
norm $\| \cdot \|_{\Omega}$, the upwind flux semi-norm $\trootonehalf\| |b\cdot
n|^{\onehalf} \jump{\cdot} \|_{\mathcal{F}_h}$, the streamline diffusion
semi-norm $\|\phi^{\onehalf} b\cdot \nabla (\cdot) \|_{\Omega}$, 
as well as the weighted $L^2$ norm 
$\trootonehalf \||b\cdot n|^{\onehalf} (\cdot)\|_{\Gamma}$,
in accordance with the a priori error estimate presented in
Section~\ref{ssec:aprior-analysis}. 
In addition, we also tabulate EOCs for the
$\| \cdot \|_{\Gamma}$ and $\|\cdot\|_{L^{\infty}(\Omega)}$
norms.

\subsubsection{Flower and popcorn domains}
\label{ssec:flower-popcorn-example}
To demonstrate the capability of the presented CutDGM to treat complex
geometries in different dimensions, we consider two test cases.  In
the first test case, the advection-reaction problem~\eqref{eq:adr-problem} is
numerically solved over a two-dimensional flower domain that is
defined by
 \begin{align}
   \Omega &= \{ (x,y) \in \RR^2 \st \phi(x,y) < 0 \} \quad \text{with }
   \Phi(x,y) = \sqrt{x^2 + y^2} - r_0 - r_1 \cos(5\atantwo(y,x)),
 \end{align}
 with $r_0 = 0.5$ and $r_1 = 0.15$, see Figure~\ref{fig:domain-flower-popcorn} (left).
 The manufactured solution $u$, the velocity field~$b$, and the reaction term $c$
 are taken from~\cite{HoustonSchwabSueli2000}
 and
 given by
 \begin{equation}
 \begin{gathered}
   u(x,y)=1+\sin(\pi(1+x)(1+y)^2/8), \quad
 b(x,y)=(0.8,0.6), \quad c = 1.0.
 \end{gathered}
   \label{eq:manufactured_uandb}
 \end{equation}
\begin{figure}[htb]
  \begin{center}
    \begin{minipage}[t]{0.49\textwidth}
    \vspace{0pt}
    \includegraphics[width=1.0\textwidth]{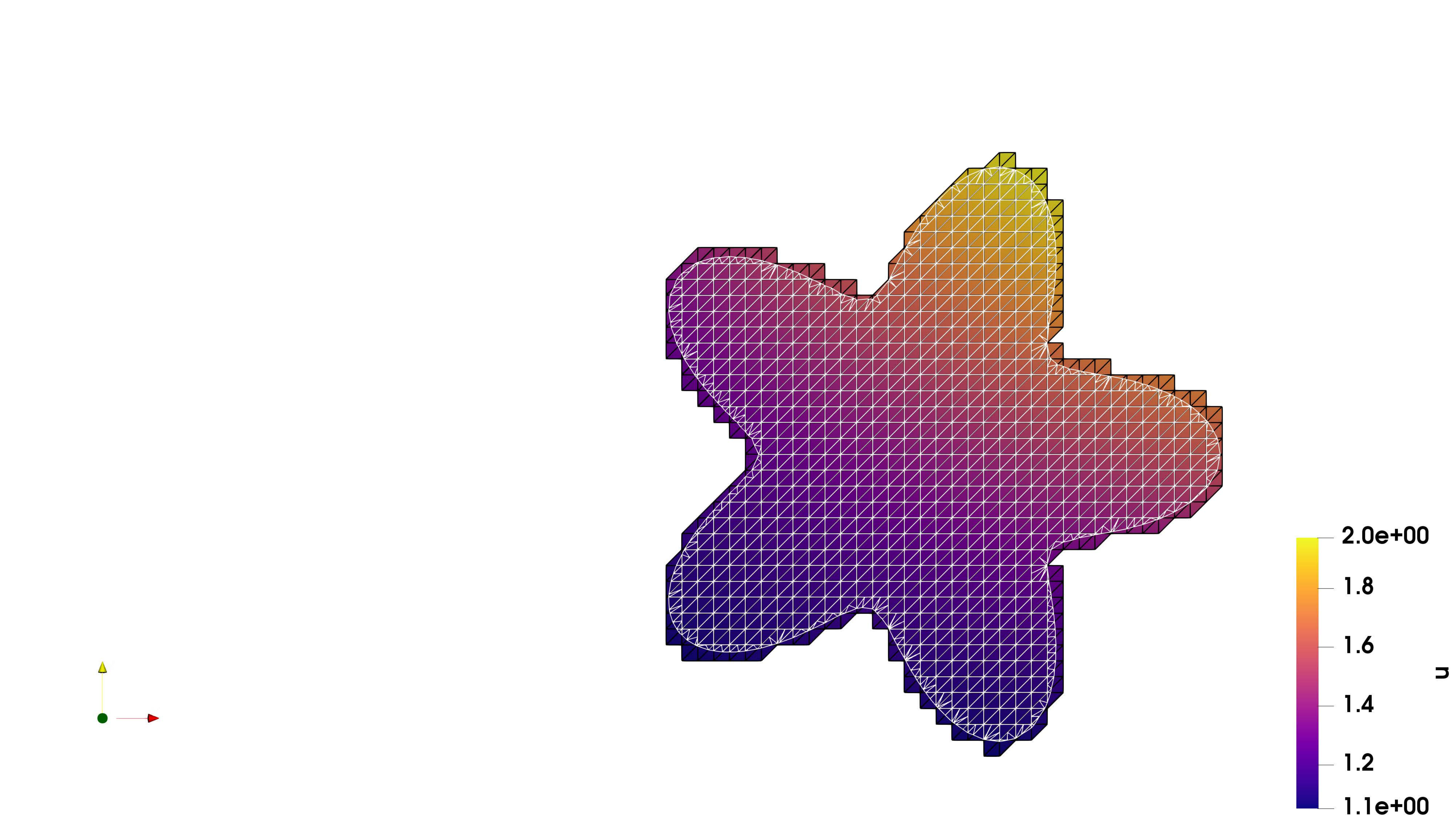}
  \end{minipage}
  \begin{minipage}[t]{0.49\textwidth}
    \vspace{0pt}
    \includegraphics[width=1.0\textwidth]{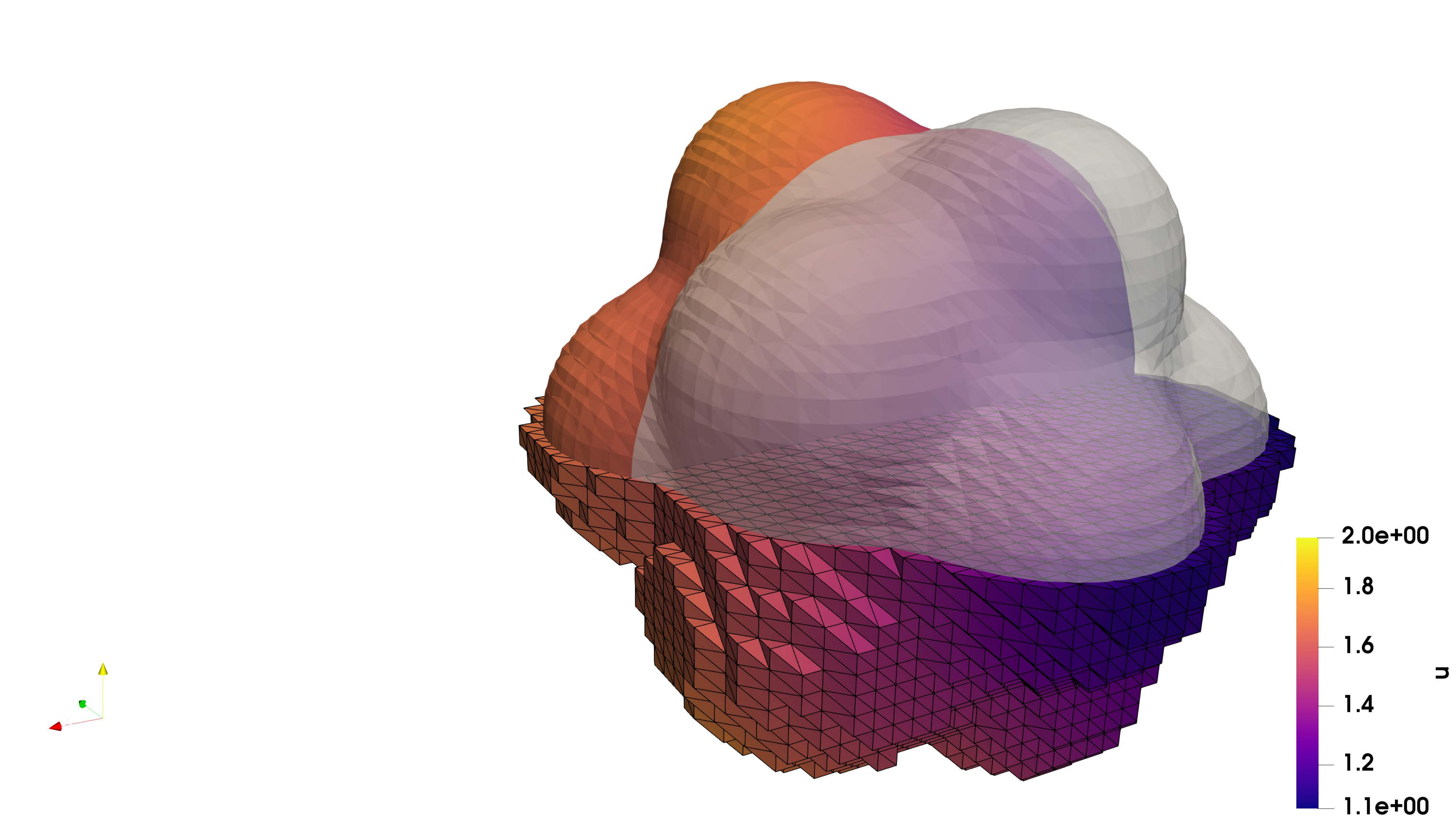}
  \end{minipage}
\end{center}
\caption{Solution plots for two dimensional flower (left) and three
  dimensional popcorn (right) examples. In both cases the solution is
  plotted over active background mesh and domain
  boundaries and physical domains are shown as
  well.}
\label{fig:domain-flower-popcorn}
\end{figure}
For the three-dimensional test case, the domain of interest is
a popcorn like geometry defined by
  \begin{equation}
  \begin{aligned}
&\Phi(x,y,z)=\sqrt{x^2+y^2+z^2}-r_0-\sum_{k=0}^{11}A \exp\big(-((x-x_k)^2+(y-y_k)^2+(z-z_k)^2)/ \sigma^2\bigr), \\
&\text{where  } \\
&(x_k,y_k,z_k)=\frac{r_0}{\sqrt{5}}\Big( 2 \cos \Big(\frac{2k \pi}{5}\Big), 2 \sin \Big(\frac{2k \pi}{5}\Big), 1  \Big), \quad 0 \leq k \leq 4, \\
&(x_k,y_k,z_k)=\frac{r_0}{\sqrt{5}}\Big( 2 \cos \Big(\frac{(2(k-5)-1) \pi}{5}\Big), 2 \sin \Big(\frac{(2(k-5)-1) \pi}{5}\Big), -1  \Big), \quad 5 \leq k \leq 9, \\
&(x_k,y_k,z_k)=(0,0,r_0), \quad k=10, \\
&(x_k,y_k,z_k)=(0,0,-r_0), \quad k=11,\\
&\text{and }\\
&r_0=0.6, \sigma=0.2, A=2,
\end{aligned}
   \label{eq:ls-popcorn}
  \end{equation}
  see Figure~\ref{fig:domain-flower-popcorn}~(right). This time, we
  set the manufactured solution~$u$, the vector field~$b$, and the
  reaction term~$c$ to
 \begin{equation}
 \begin{gathered}
 u(x,y,z)=1+\sin(\pi(1+x)(1+y^2)(1+z^3)/8), \quad
 b(x,y,z)=(0.8,0.6,0.4), \quad c = 1.0.
 \end{gathered}
   \label{eq:manufactured_uandb_popcorn}
 \end{equation}
 For both examples we employ linear elements together with stabilization
 parameters $\gamma_b^f=\gamma_c^f=0.01$.
 The calculated EOCs are summarized in~Table~\ref{tab:convanalysis_2D_3d} and confirm
 the theoretically expected convergence rates of
 $\mcO(h^{\threehalf})$. 
 For the $L^2$ norm, we even observe second-order convergence
 for these particular examples and mesh definitions, but we point out that this cannot be expected
 in general, see~\cite{Peterson1991,Burman2005}.
 The computed convergence rate orders
 using the $\|\cdot\|_{\Gamma}$ and $\| \cdot \|_{L^{\infty}}$ norms 
 are close to $2$ and $1.5$, respectively. 

 Finally, we repeat the numerical experiment with $\PP_0$ elements,
 but this time, we disable any ghost penalty stabilization.
 The resulting errors for the 2D flower domain are reported in Table~\ref{tab:convanalysis_2D_3d_p0},
 with rates between $\mcO(h^{\onehalf})$
 and $\mcO(h)$,
 confirming the convergence rates
 predicted by Proposition~\ref{prop:apriori-est-p0}.
 The experimentally observed convergence rates for the 3D popcorn domain
 were very similar and are not reported here.

 \begin{table}[htb]
   \small
   \begin{center}
    \caption{Convergence rates for the 2D flower (top) and 3D popcorn (bottom) example using
      $\PP_1(\mcT_k)$.}
   \label{tab:convanalysis_2D_3d}
     \begin {tabular}{cr<{\pgfplotstableresetcolortbloverhangright }@{}l<{\pgfplotstableresetcolortbloverhangleft }cr<{\pgfplotstableresetcolortbloverhangright }@{}l<{\pgfplotstableresetcolortbloverhangleft }cr<{\pgfplotstableresetcolortbloverhangright }@{}l<{\pgfplotstableresetcolortbloverhangleft }c}%
\toprule Level $k$&\multicolumn {2}{c}{$\|e_k\|_{\Omega }$}&EOC&\multicolumn {2}{c}{$\trootonehalf \| |b\cdot n|^{\onehalf } \jump {e_k} \|_{\mathcal {F}_h}$}&EOC&\multicolumn {2}{c}{$\|\phi ^{\onehalf } b\cdot \nabla e_k \|_{\Omega }$}&EOC\\\midrule %
\pgfutilensuremath {0}&$8.69$&$\cdot 10^{-3}$&--&$1.17$&$\cdot 10^{-2}$&--&$1.82$&$\cdot 10^{-2}$&--\\%
\pgfutilensuremath {1}&$2.05$&$\cdot 10^{-3}$&\pgfutilensuremath {2.09}&$3.91$&$\cdot 10^{-3}$&\pgfutilensuremath {1.58}&$4.76$&$\cdot 10^{-3}$&\pgfutilensuremath {1.93}\\%
\pgfutilensuremath {2}&$5.25$&$\cdot 10^{-4}$&\pgfutilensuremath {1.96}&$1.30$&$\cdot 10^{-3}$&\pgfutilensuremath {1.59}&$1.64$&$\cdot 10^{-3}$&\pgfutilensuremath {1.53}\\%
\pgfutilensuremath {3}&$1.34$&$\cdot 10^{-4}$&\pgfutilensuremath {1.97}&$4.44$&$\cdot 10^{-4}$&\pgfutilensuremath {1.55}&$5.40$&$\cdot 10^{-4}$&\pgfutilensuremath {1.60}\\%
\pgfutilensuremath {4}&$3.33$&$\cdot 10^{-5}$&\pgfutilensuremath {2.01}&$1.56$&$\cdot 10^{-4}$&\pgfutilensuremath {1.51}&$2.02$&$\cdot 10^{-4}$&\pgfutilensuremath {1.42}\\%
\pgfutilensuremath {5}&$8.40$&$\cdot 10^{-6}$&\pgfutilensuremath {1.98}&$5.42$&$\cdot 10^{-5}$&\pgfutilensuremath {1.52}&$6.89$&$\cdot 10^{-5}$&\pgfutilensuremath {1.55}\\\midrule %
Level $k$&\multicolumn {2}{c}{$\trootonehalf \||b\cdot n|^{\onehalf } e_k \|_{\Gamma }$}&EOC&\multicolumn {2}{c}{$\|e_k\|_{L^{2}(\Gamma )}$}&EOC&\multicolumn {2}{c}{$\|e_k\|_{L^{\infty }(\Omega )}$}&EOC\\\midrule %
\pgfutilensuremath {0}&$1.20$&$\cdot 10^{-2}$&--&$2.55$&$\cdot 10^{-2}$&--&$4.29$&$\cdot 10^{-2}$&--\\%
\pgfutilensuremath {1}&$2.51$&$\cdot 10^{-3}$&\pgfutilensuremath {2.26}&$4.74$&$\cdot 10^{-3}$&\pgfutilensuremath {2.43}&$7.64$&$\cdot 10^{-3}$&\pgfutilensuremath {2.49}\\%
\pgfutilensuremath {2}&$7.00$&$\cdot 10^{-4}$&\pgfutilensuremath {1.84}&$1.26$&$\cdot 10^{-3}$&\pgfutilensuremath {1.91}&$3.43$&$\cdot 10^{-3}$&\pgfutilensuremath {1.15}\\%
\pgfutilensuremath {3}&$1.80$&$\cdot 10^{-4}$&\pgfutilensuremath {1.96}&$3.34$&$\cdot 10^{-4}$&\pgfutilensuremath {1.92}&$1.24$&$\cdot 10^{-3}$&\pgfutilensuremath {1.47}\\%
\pgfutilensuremath {4}&$4.51$&$\cdot 10^{-5}$&\pgfutilensuremath {1.99}&$8.81$&$\cdot 10^{-5}$&\pgfutilensuremath {1.92}&$3.16$&$\cdot 10^{-4}$&\pgfutilensuremath {1.97}\\%
\pgfutilensuremath {5}&$1.14$&$\cdot 10^{-5}$&\pgfutilensuremath {1.99}&$2.29$&$\cdot 10^{-5}$&\pgfutilensuremath {1.94}&$1.03$&$\cdot 10^{-4}$&\pgfutilensuremath {1.62}\\\bottomrule %
\end {tabular}%

     \\[3ex]
    \begin {tabular}{cr<{\pgfplotstableresetcolortbloverhangright }@{}l<{\pgfplotstableresetcolortbloverhangleft }cr<{\pgfplotstableresetcolortbloverhangright }@{}l<{\pgfplotstableresetcolortbloverhangleft }cr<{\pgfplotstableresetcolortbloverhangright }@{}l<{\pgfplotstableresetcolortbloverhangleft }c}%
\toprule Level $k$&\multicolumn {2}{c}{$\|e_k\|_{\Omega }$}&EOC&\multicolumn {2}{c}{$\trootonehalf \| |b\cdot n|^{\onehalf } \jump {e_k} \|_{\mathcal {F}_h}$}&EOC&\multicolumn {2}{c}{$\|\phi ^{\onehalf } b\cdot \nabla e_k \|_{\Omega }$}&EOC\\\midrule %
\pgfutilensuremath {0}&$1.86$&$\cdot 10^{-2}$&--&$1.82$&$\cdot 10^{-2}$&--&$3.48$&$\cdot 10^{-2}$&--\\%
\pgfutilensuremath {1}&$4.41$&$\cdot 10^{-3}$&\pgfutilensuremath {2.07}&$7.35$&$\cdot 10^{-3}$&\pgfutilensuremath {1.31}&$1.10$&$\cdot 10^{-2}$&\pgfutilensuremath {1.66}\\%
\pgfutilensuremath {2}&$1.07$&$\cdot 10^{-3}$&\pgfutilensuremath {2.05}&$2.58$&$\cdot 10^{-3}$&\pgfutilensuremath {1.51}&$3.11$&$\cdot 10^{-3}$&\pgfutilensuremath {1.83}\\%
\pgfutilensuremath {3}&$2.62$&$\cdot 10^{-4}$&\pgfutilensuremath {2.02}&$8.29$&$\cdot 10^{-4}$&\pgfutilensuremath {1.64}&$1.11$&$\cdot 10^{-3}$&\pgfutilensuremath {1.49}\\%
\pgfutilensuremath {4}&$6.54$&$\cdot 10^{-5}$&\pgfutilensuremath {2.00}&$2.73$&$\cdot 10^{-4}$&\pgfutilensuremath {1.60}&$3.54$&$\cdot 10^{-4}$&\pgfutilensuremath {1.65}\\\midrule %
Level $k$&\multicolumn {2}{c}{$\trootonehalf \||b\cdot n|^{\onehalf } e_k \|_{\Gamma }$}&EOC&\multicolumn {2}{c}{$\|e_k\|_{L^{2}(\Gamma )}$}&EOC&\multicolumn {2}{c}{$\|e_k\|_{L^{\infty }(\Omega )}$}&EOC\\\midrule %
\pgfutilensuremath {0}&$1.75$&$\cdot 10^{-2}$&--&$3.74$&$\cdot 10^{-2}$&--&$4.55$&$\cdot 10^{-2}$&--\\%
\pgfutilensuremath {1}&$4.62$&$\cdot 10^{-3}$&\pgfutilensuremath {1.92}&$9.17$&$\cdot 10^{-3}$&\pgfutilensuremath {2.03}&$1.27$&$\cdot 10^{-2}$&\pgfutilensuremath {1.84}\\%
\pgfutilensuremath {2}&$1.18$&$\cdot 10^{-3}$&\pgfutilensuremath {1.97}&$2.34$&$\cdot 10^{-3}$&\pgfutilensuremath {1.97}&$3.94$&$\cdot 10^{-3}$&\pgfutilensuremath {1.69}\\%
\pgfutilensuremath {3}&$3.02$&$\cdot 10^{-4}$&\pgfutilensuremath {1.96}&$6.17$&$\cdot 10^{-4}$&\pgfutilensuremath {1.92}&$1.15$&$\cdot 10^{-3}$&\pgfutilensuremath {1.77}\\%
\pgfutilensuremath {4}&$7.70$&$\cdot 10^{-5}$&\pgfutilensuremath {1.97}&$1.65$&$\cdot 10^{-4}$&\pgfutilensuremath {1.90}&$3.63$&$\cdot 10^{-4}$&\pgfutilensuremath {1.67}\\\bottomrule %
\end {tabular}%

   \end{center}
 \end{table}

 \begin{table}[htb]
   \small
   \begin{center}
    \caption{Convergence rates for the 2D flower  example using
      $\PP_0(\mcT_k)$.}
   \label{tab:convanalysis_2D_3d_p0}
     \begin {tabular}{cr<{\pgfplotstableresetcolortbloverhangright }@{}l<{\pgfplotstableresetcolortbloverhangleft }cr<{\pgfplotstableresetcolortbloverhangright }@{}l<{\pgfplotstableresetcolortbloverhangleft }cr<{\pgfplotstableresetcolortbloverhangright }@{}l<{\pgfplotstableresetcolortbloverhangleft }c}%
\toprule Level $N$&\multicolumn {2}{c}{$\|e_k\|_{\Omega }$}&EOC&\multicolumn {2}{c}{$\trootonehalf \| |b\cdot n|^{\onehalf } \jump {e_k} \|_{\mathcal {F}_h}$}&EOC&\multicolumn {2}{c}{$\|\phi ^{\onehalf } b\cdot \nabla e_k \|_{\Omega }$}&EOC\\\midrule %
\pgfutilensuremath {0}&$5.43$&$\cdot 10^{-2}$&--&$1.37$&$\cdot 10^{-1}$&--&$2.08$&$\cdot 10^{-1}$&--\\%
\pgfutilensuremath {1}&$3.27$&$\cdot 10^{-2}$&\pgfutilensuremath {0.73}&$1.04$&$\cdot 10^{-1}$&\pgfutilensuremath {0.40}&$1.37$&$\cdot 10^{-1}$&\pgfutilensuremath {0.61}\\%
\pgfutilensuremath {2}&$1.60$&$\cdot 10^{-2}$&\pgfutilensuremath {1.03}&$7.65$&$\cdot 10^{-2}$&\pgfutilensuremath {0.45}&$9.22$&$\cdot 10^{-2}$&\pgfutilensuremath {0.57}\\%
\pgfutilensuremath {3}&$8.38$&$\cdot 10^{-3}$&\pgfutilensuremath {0.93}&$5.42$&$\cdot 10^{-2}$&\pgfutilensuremath {0.50}&$6.28$&$\cdot 10^{-2}$&\pgfutilensuremath {0.55}\\%
\pgfutilensuremath {4}&$4.13$&$\cdot 10^{-3}$&\pgfutilensuremath {1.02}&$3.86$&$\cdot 10^{-2}$&\pgfutilensuremath {0.49}&$4.52$&$\cdot 10^{-2}$&\pgfutilensuremath {0.47}\\%
\pgfutilensuremath {5}&$2.19$&$\cdot 10^{-3}$&\pgfutilensuremath {0.91}&$2.73$&$\cdot 10^{-2}$&\pgfutilensuremath {0.50}&$3.10$&$\cdot 10^{-2}$&\pgfutilensuremath {0.55}\\\midrule %
Level $k$&\multicolumn {2}{c}{$\trootonehalf \||b\cdot n|^{\onehalf } e_k \|_{\Gamma }$}&EOC&\multicolumn {2}{c}{$\|e_k\|_{L^{2}(\Gamma )}$}&EOC&\multicolumn {2}{c}{$\|e_k\|_{L^{\infty }(\Omega )}$}&EOC\\\midrule %
\pgfutilensuremath {0}&$9.87$&$\cdot 10^{-2}$&--&$2.03$&$\cdot 10^{-1}$&--&$3.12$&$\cdot 10^{-1}$&--\\%
\pgfutilensuremath {1}&$6.03$&$\cdot 10^{-2}$&\pgfutilensuremath {0.71}&$1.21$&$\cdot 10^{-1}$&\pgfutilensuremath {0.74}&$2.21$&$\cdot 10^{-1}$&\pgfutilensuremath {0.50}\\%
\pgfutilensuremath {2}&$2.87$&$\cdot 10^{-2}$&\pgfutilensuremath {1.07}&$5.41$&$\cdot 10^{-2}$&\pgfutilensuremath {1.16}&$1.01$&$\cdot 10^{-1}$&\pgfutilensuremath {1.13}\\%
\pgfutilensuremath {3}&$1.54$&$\cdot 10^{-2}$&\pgfutilensuremath {0.90}&$3.17$&$\cdot 10^{-2}$&\pgfutilensuremath {0.77}&$7.17$&$\cdot 10^{-2}$&\pgfutilensuremath {0.49}\\%
\pgfutilensuremath {4}&$7.87$&$\cdot 10^{-3}$&\pgfutilensuremath {0.97}&$1.77$&$\cdot 10^{-2}$&\pgfutilensuremath {0.84}&$4.81$&$\cdot 10^{-2}$&\pgfutilensuremath {0.58}\\%
\pgfutilensuremath {5}&$4.10$&$\cdot 10^{-3}$&\pgfutilensuremath {0.94}&$9.88$&$\cdot 10^{-3}$&\pgfutilensuremath {0.84}&$2.64$&$\cdot 10^{-2}$&\pgfutilensuremath {0.87}\\\bottomrule %
\end {tabular}%

   \end{center}
 \end{table}

 \subsubsection{Wavy inflow and outflow boundary}
 \label{ssec:wavy-boundary-example}
Next, we study the performance of the proposed CutDG method
using high order elements for smooth solutions, and at the presence of
a sharp internal layer in the manufactured solution.
To cover both scenarios within a single problem set-up, we manufacture
an analytical solution $u_{\epsilon}$ depending on a parameter
$\epsilon$ which results in a smooth function if $\epsilon \sim 1$,
while choosing $\epsilon \ll 1$ 
produces a solution with a sharp internal layer.
In the latter case, the manufactured solution is practically discontinuous
whenever the internal layer cannot be resolved by the mesh.
In both test cases, 
the problem~\eqref{eq:adr-problem} is  solved on the domain
\begin{gather}
  \Omega = [-1.0,1.0]^2 \cap \{  \phi^+ < 0.85 \} \cap \{  \phi^- > -0.85 \},
  \intertext{where}
  \phi^+(x,y)=y-0.1\cos(8x),
  \qquad
  \phi^-(x,y)=y+0.1\cos(8x).
\end{gather}
The velocity field~$b$ and the analytical solution~$u_{\epsilon}$ are set to
\begin{equation}
  \begin{aligned} 
      u_{\epsilon}(x,y) &= \exp(\lambda(x,y) \arcsin ( (1-x^2) \pi \cos(2 \pi y)/25 ) )\arctan(\lambda(x,y) / \epsilon), && \quad\\
      b(x,y) &=((1-x^2) \pi \cos(2 \pi y)/25  ,(1+4x^2)/5 ),  &&\quad     
  \end{aligned}
  \label{eq:analytical-soln}
\end{equation}
where $\lambda(x,y)=(x-0.1\sin(2 \pi y))/5$.
It is also important to realize that
$\lambda$ defines the shape of the internal layer.
Thanks to the definition of the velocity field $b$, we observe that the inflow and outflow
boundaries are defined by the level sets 
\begin{equation}
 \Gamma^-  = \{  \phi^- = -0.85 \} \cap \Omega, \qquad
 \Gamma^+  = \{  \phi^+ = 0.85 \} \cap \Omega,
\label{eq:ls}
\end{equation}
and for $\epsilon \ll 1$, the streamlines of the velocity field $b$
are parallel to the discontinuity, see Figure \ref{fig:domain-wavy}
(bottom).
For our convergence tests, we 
study two extreme cases where we set $\epsilon=1.0$ first and
$\epsilon=10^{-8}$ later.  Recalling the a priori
error analysis from Sections~\ref{ssec:aprior-analysis}, the expected
convergence rate for smooth solutions is $\mcO(h^{p+\onehalf})$.  For the
discontinuous case, one can only expect a convergence order of
at most $\mcO(h^{\nicefrac{1}{2}})$ in the (global) $L^2$ norm,
independent of the chosen polynomial order $p$.
Consequently, we
consider only the case $p=1$ for the nearly discontinuous solution.
\begin{figure}[htb]
  \begin{subfigure}[t]{1.0\textwidth}
  \begin{center}
    \begin{minipage}[t]{0.49\textwidth}
    \vspace{0pt}
    \includegraphics[width=1.0\textwidth]{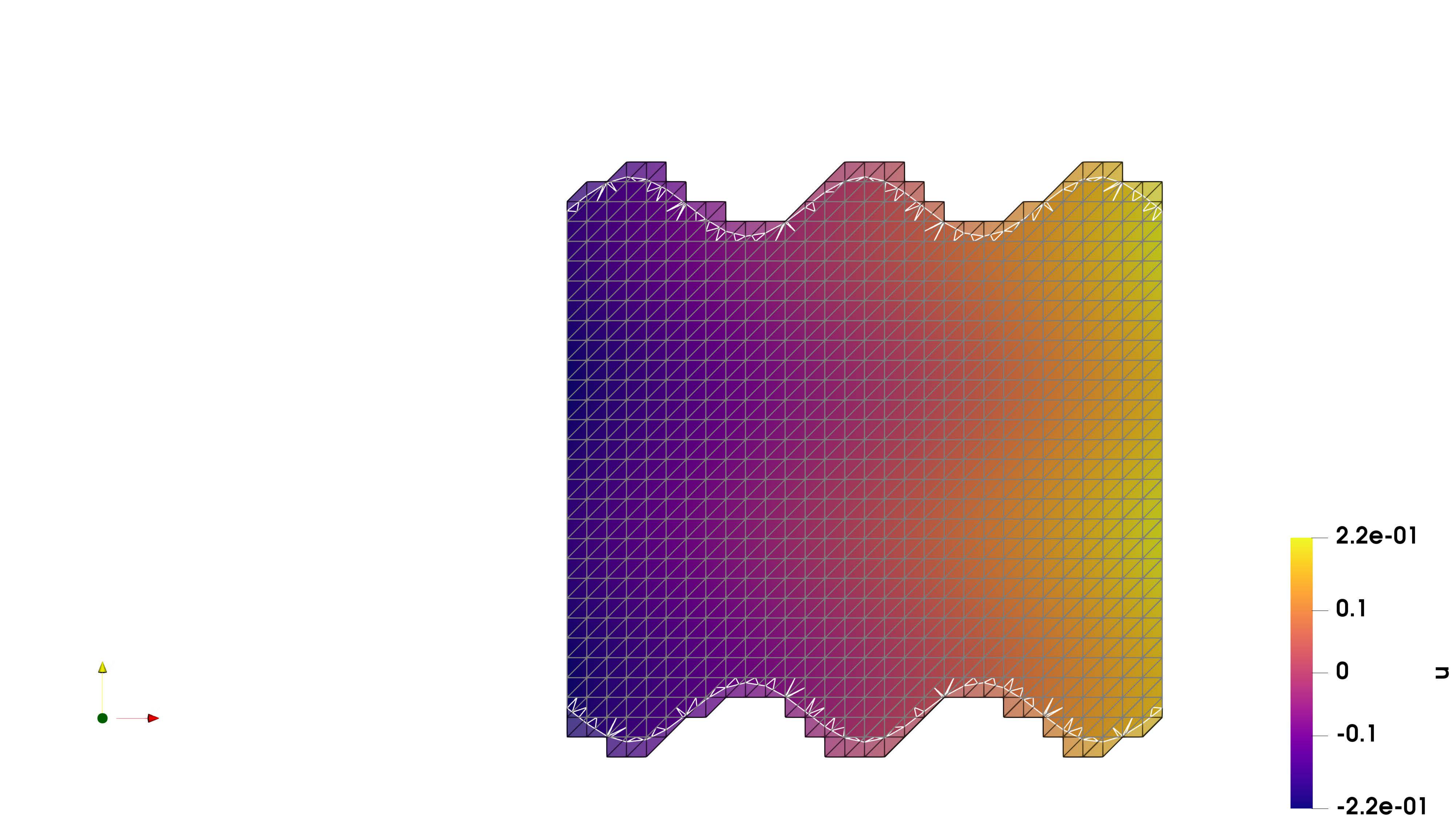}
  \end{minipage}
  \begin{minipage}[t]{0.49\textwidth}
    \vspace{0pt}
    \includegraphics[width=1.0\textwidth]{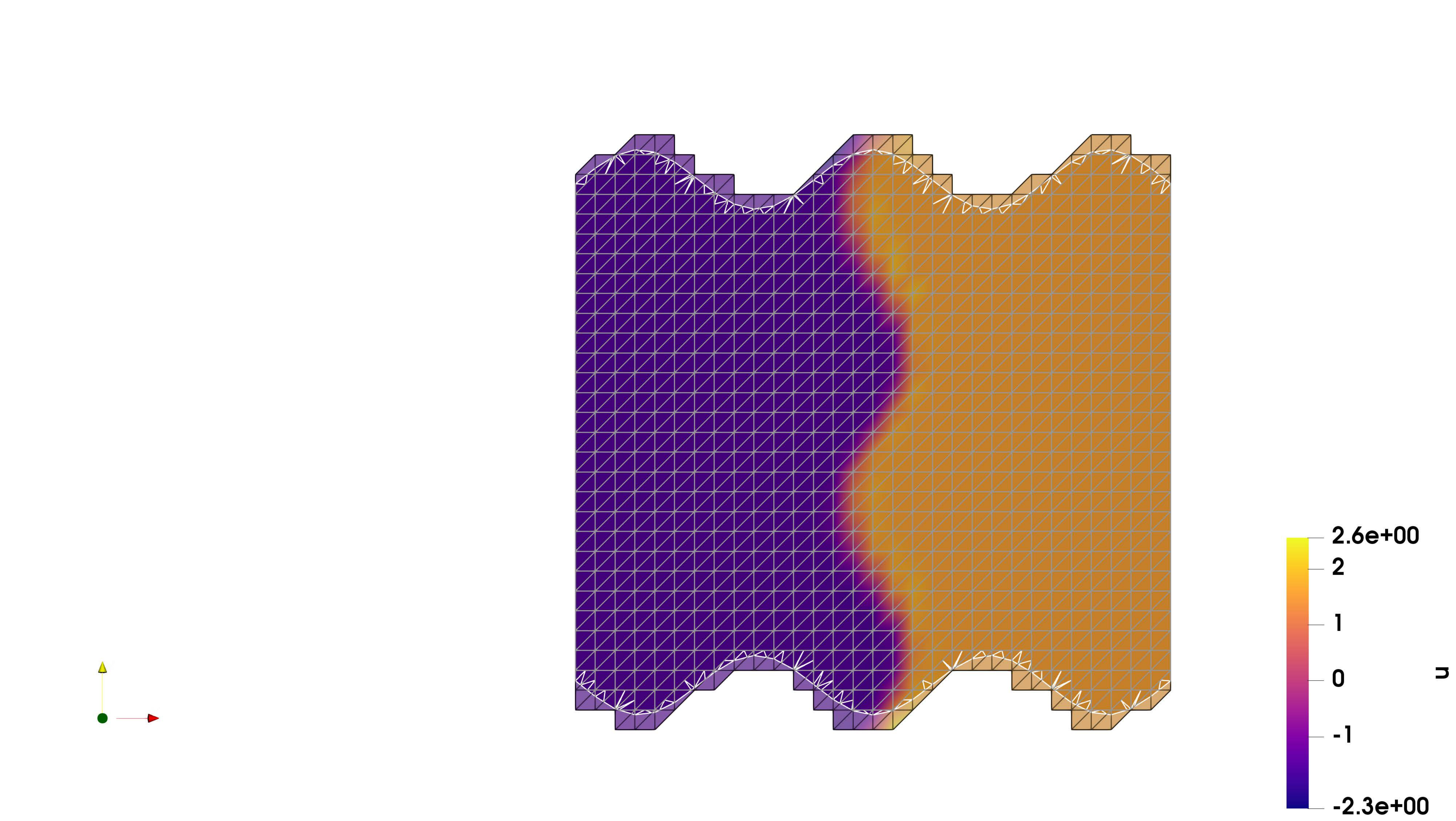}
  \end{minipage}
\end{center}
\end{subfigure}
  \begin{subfigure}[t]{1.0\textwidth}
   \hspace{8.7pt}\includegraphics[width=0.824\textwidth]{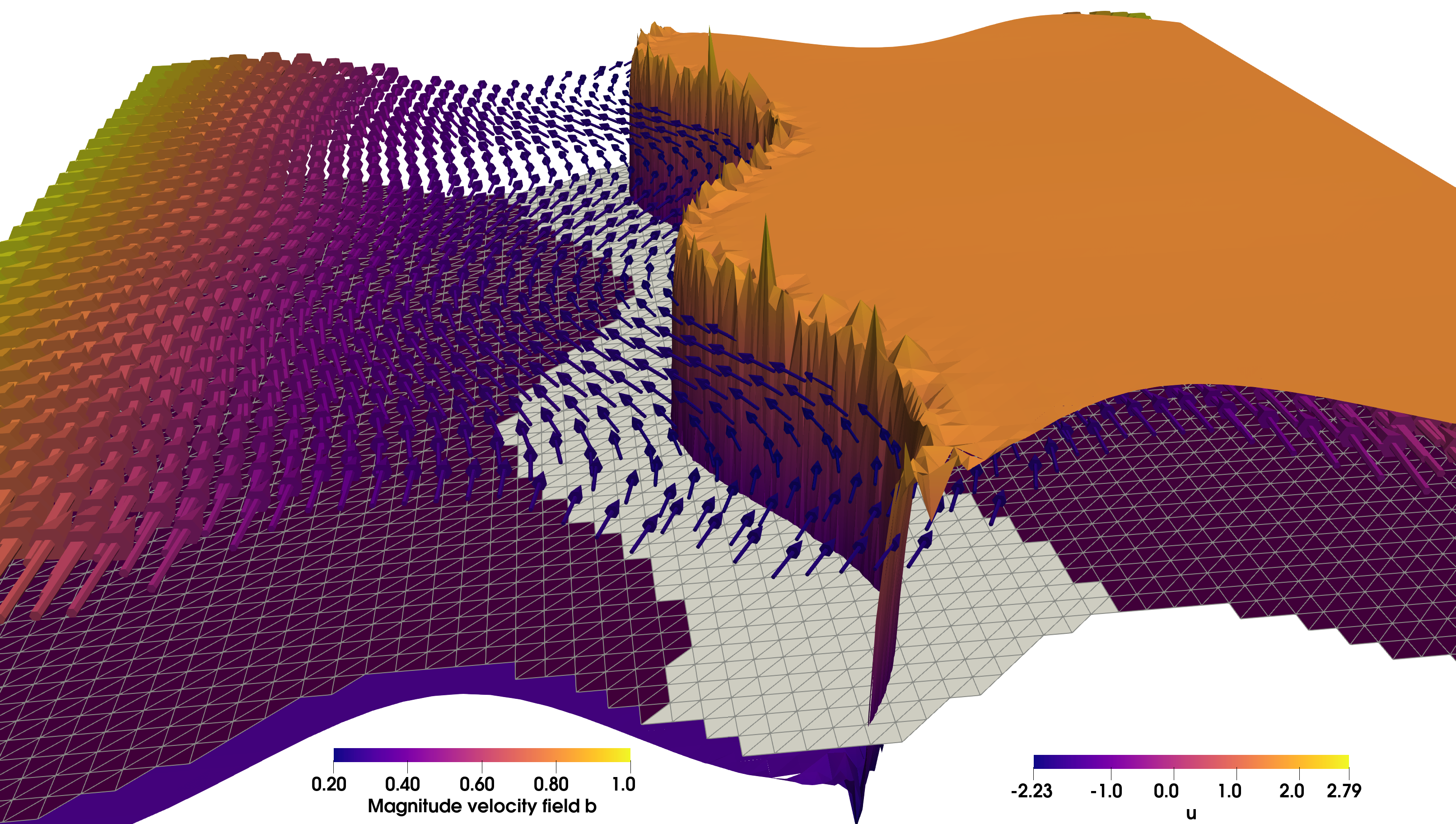}
\end{subfigure}
\caption{Solution plots for wavy inflow and outflow boundary
example. (Top/left) Smooth solution with $\epsilon=1$.
(Top/right) Nearly discontinuous solution with $\epsilon=10^{-8}$ . In both
cases the solution is plotted over the active background mesh
$\mathcal{T}_k$, the wavy inflow and outflow boundaries together
with physical domain are shown as well.
(Bottom) Close-up of the graph of the computed approximation to the rough solution near the
discontinuity, together with the velocity field $b$. Grey mesh domain around
the discontinuity was disregarded in modified EOC computations where only  the
error away from the internal layer was considered.}
  \label{fig:domain-wavy}
\end{figure}

Now setting $\epsilon=1$ in~\eqref{eq:analytical-soln}, we obtain a smooth solution
see Figure \ref{fig:domain-wavy} (left). Employing
high order elements up to $p=3$, optimal convergence rates are obtained, see 
Figure~\ref{fig:if:convergence_plots}. As before, we observe an EOC of $\mcO(h^{p+1})$
rather than $\mcO(h^{p+\onehalf})$ in the $L^2$ norm.
\begin{figure}[htb]
  \begin{subfigure}[t]{1.0\textwidth}
  \begin{center}
    \includegraphics[width=0.32\textwidth,page=1]{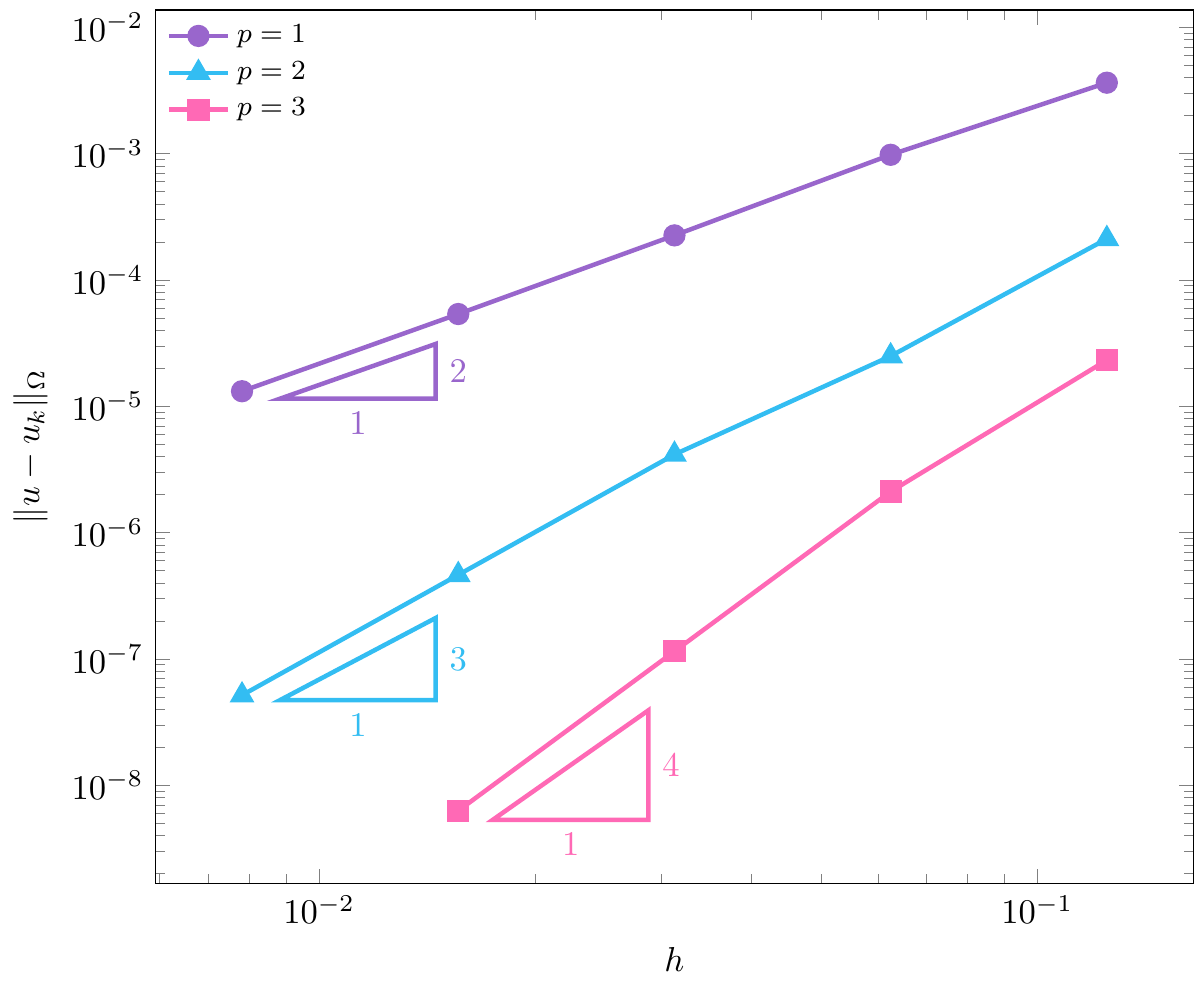}
    \includegraphics[width=0.32\textwidth,page=2]{figures/convergence_plots_hp_p1_to_p3_example_wavy_boundary.pdf} 
    \includegraphics[width=0.32\textwidth,page=3]{figures/convergence_plots_hp_p1_to_p3_example_wavy_boundary.pdf}
  \end{center}
  \caption{Convergence rates in the $L^2(\Omega)$ (left), upwind flux (middle), and streamline
    diffusion (right) norm.}
\end{subfigure}
  \begin{subfigure}[t]{1.0\textwidth}
  \begin{center}
    \includegraphics[width=0.32\textwidth,page=4]{figures/convergence_plots_hp_p1_to_p3_example_wavy_boundary.pdf}
    \includegraphics[width=0.32\textwidth,page=5]{figures/convergence_plots_hp_p1_to_p3_example_wavy_boundary.pdf} 
    \includegraphics[width=0.32\textwidth,page=6]{figures/convergence_plots_hp_p1_to_p3_example_wavy_boundary.pdf}
  \caption{Convergence rates in the weighted $L^2(\Gamma)$ (left), standard $L^2(\Gamma)$ (middle), and $L^{\infty}(\Omega)$ (right) norm.}
  \end{center}
\end{subfigure}
\vspace{-2ex}
  \caption{
    Observed error convergence rates for wavy inflow and
    outflow example with $p=1,2,3$. For $p=3$, computer memory restrictions
    prohibited to perform more than three mesh refinements.}
\label{fig:if:convergence_plots}
\end{figure}
Next, we set $\epsilon=10^{-8}$ in \eqref{eq:analytical-soln},
where we observe a sharp internal layer in the middle of the
domain, see Figure~\ref{fig:domain-wavy} (right). For our mesh
resolution, this layer can be considered as discontinuous and the
tabulated results in Table~\ref{tab:convanalysis_if_3d_P1} (top) confirm
the reduction of the convergence order to
$\mcO(h^{\nicefrac{1}{2}})$ in the $L^2$ norm, and with no
convergence in the other considered error norms.
A closer look at the computed solution 
in Figure~\ref{fig:domain-wavy} (right)
shows the typical osillations triggered by the Gibbs phenomenon
which, as in the standard fitted upwind DG method,
only appear in the close vicinity of the layer.
For completeness,
we repeat the numerical experiment
but this time, we confine the error computations
to those elements and faces which are contained in 
$\Omega_{\delta} = \Omega \setminus 
\{(x,y) \in \RR^2 \st |\lambda(x,y)| > 5 \delta \}$
with $\delta = 0.25$.
The exclusion of
a $\delta$ tubular neighborhood around the sharp layer
restores the previous optimal convergence rates as indicated by
results displayed in Table~\ref{tab:convanalysis_if_3d_P1}
(bottom). 
How to combine the proposed stabilized CutDG framework with various
shock capturing or limiter techiques~\cite{Shu2009}
to control spurious oscillation near discontinuities
will be part of our future research.

\begin{table}[htb]
  \small
  \begin{center}
    \caption{Convergence rates for the third example with a sharp
    internal layer for $\epsilon = 10^{-8}$ using $\PP_1(\mcT_k)$.
    Error norms are computed on whole domain (top) and 
    outside a $\delta$ neigborhood of the transition layer 
    with $\delta = 0.25$ (bottom).
    }
  \label{tab:convanalysis_if_3d_P1}
    \begin {tabular}{cr<{\pgfplotstableresetcolortbloverhangright }@{}l<{\pgfplotstableresetcolortbloverhangleft }cr<{\pgfplotstableresetcolortbloverhangright }@{}l<{\pgfplotstableresetcolortbloverhangleft }cr<{\pgfplotstableresetcolortbloverhangright }@{}l<{\pgfplotstableresetcolortbloverhangleft }c}%
\toprule Level $k$&\multicolumn {2}{c}{$\|e_k\|_{\Omega }$}&EOC&\multicolumn {2}{c}{$\trootonehalf \| |b\cdot n|^{\onehalf } \jump {e_k} \|_{\mathcal {F}_h}$}&EOC&\multicolumn {2}{c}{$\|\phi ^{\onehalf } b\cdot \nabla e_k \|_{\Omega }$}&EOC\\\midrule %
\pgfutilensuremath {0}&$8.68$&$\cdot 10^{-1}$&--&$1.30$&$\cdot 10^{-1}$&--&$9.05$&$\cdot 10^{-1}$&--\\%
\pgfutilensuremath {1}&$5.83$&$\cdot 10^{-1}$&\pgfutilensuremath {0.57}&$1.31$&$\cdot 10^{-1}$&\pgfutilensuremath {-0.01}&$1.12$&$\cdot 10^{0}$&\pgfutilensuremath {-0.31}\\%
\pgfutilensuremath {2}&$4.05$&$\cdot 10^{-1}$&\pgfutilensuremath {0.52}&$1.27$&$\cdot 10^{-1}$&\pgfutilensuremath {0.04}&$1.08$&$\cdot 10^{0}$&\pgfutilensuremath {0.06}\\%
\pgfutilensuremath {3}&$2.83$&$\cdot 10^{-1}$&\pgfutilensuremath {0.52}&$9.80$&$\cdot 10^{-2}$&\pgfutilensuremath {0.37}&$1.00$&$\cdot 10^{0}$&\pgfutilensuremath {0.10}\\%
\pgfutilensuremath {4}&$2.14$&$\cdot 10^{-1}$&\pgfutilensuremath {0.40}&$7.52$&$\cdot 10^{-2}$&\pgfutilensuremath {0.38}&$1.00$&$\cdot 10^{0}$&\pgfutilensuremath {0.00}\\\midrule %
Level $k$&\multicolumn {2}{c}{$\|e_k\|_{\Omega }$}&EOC&\multicolumn {2}{c}{$\trootonehalf \| |b\cdot n|^{\onehalf } \jump {e_k} \|_{\mathcal {F}_h}$}&EOC&\multicolumn {2}{c}{$\|\phi ^{\onehalf } b\cdot \nabla e_k \|_{\Omega }$}&EOC\\\midrule %
\pgfutilensuremath {0}&$6.96$&$\cdot 10^{-1}$&--&$1.06$&$\cdot 10^{-1}$&--&$8.15$&$\cdot 10^{-1}$&--\\%
\pgfutilensuremath {1}&$3.10$&$\cdot 10^{-2}$&\pgfutilensuremath {4.49}&$2.25$&$\cdot 10^{-2}$&\pgfutilensuremath {2.24}&$1.49$&$\cdot 10^{-2}$&\pgfutilensuremath {5.77}\\%
\pgfutilensuremath {2}&$3.78$&$\cdot 10^{-4}$&\pgfutilensuremath {6.36}&$5.32$&$\cdot 10^{-4}$&\pgfutilensuremath {5.40}&$9.61$&$\cdot 10^{-4}$&\pgfutilensuremath {3.95}\\%
\pgfutilensuremath {3}&$7.62$&$\cdot 10^{-5}$&\pgfutilensuremath {2.31}&$1.71$&$\cdot 10^{-4}$&\pgfutilensuremath {1.64}&$2.98$&$\cdot 10^{-4}$&\pgfutilensuremath {1.69}\\%
\pgfutilensuremath {4}&$1.87$&$\cdot 10^{-5}$&\pgfutilensuremath {2.03}&$6.05$&$\cdot 10^{-5}$&\pgfutilensuremath {1.50}&$1.01$&$\cdot 10^{-4}$&\pgfutilensuremath {1.56}\\\bottomrule %
\end {tabular}%

  \end{center}
\end{table}


\subsection{Geometrical robustness}
\label{sec:effectofghostpenalties}
Further, we numerically investigate the geometrical
robustness of the derived a priori error and condition number estimates,
illustrating the importance and different roles of the ghost penalties $g_c$
and $g_b$. Our experiments are based on the following setup. We start from a
structured background mesh $\widetilde{\mcT}_h$ for the rectangular domain
$\widetilde{\Omega} = [-0.35,0.35]^2\subset \RR^2$ with mesh size $h = 0.7/N$
and $N=10$. To generate potentially critical cut configurations, we
define a family $\{\Omega_{\delta_k}\}_{k=1}^{1000}$ of translated
circular domains $\Omega_{\delta_k}= \Omega + \delta_k \bft$ with
initial domain  
$\Omega = \{ (x,y) \in \RR^2 \st x^2 + y^2 - 0.25^2 < 0\}$
and set the translation vector~$\bft$ and translation parameter~$\delta_k$ to
$\tfrac{1}{\sqrt{2}} (h,h)$ and $k/1000$, respectively.

\subsubsection{Sensitivity of the approximation error}
\label{ssec:sensitivity-approx-errors}
In our first robustness test, we compute the
$\| \cdot \|_{\Omega}$ and
$\|\phi_b^{\onehalf } \bnabla (\cdot)\|_{\Omega}$ error
of the method \eqref{eq:cutDGM} as a function
of the translation parameter $\delta$.
The manufactured solution, velocity field and reaction term
are again given by~\eqref{eq:manufactured_uandb}.
To demonstrate that the ghost penalties are necessary
to render the errors insensitive to the particular cut configuration,
we repeat the computation with either $g_c$ or $g_b$ or both turned off,
setting the corresponding stabilization parameters to $0$.
\begin{figure}[ht!]
  \begin{center}
      \includegraphics[width=0.32\textwidth,page=1]{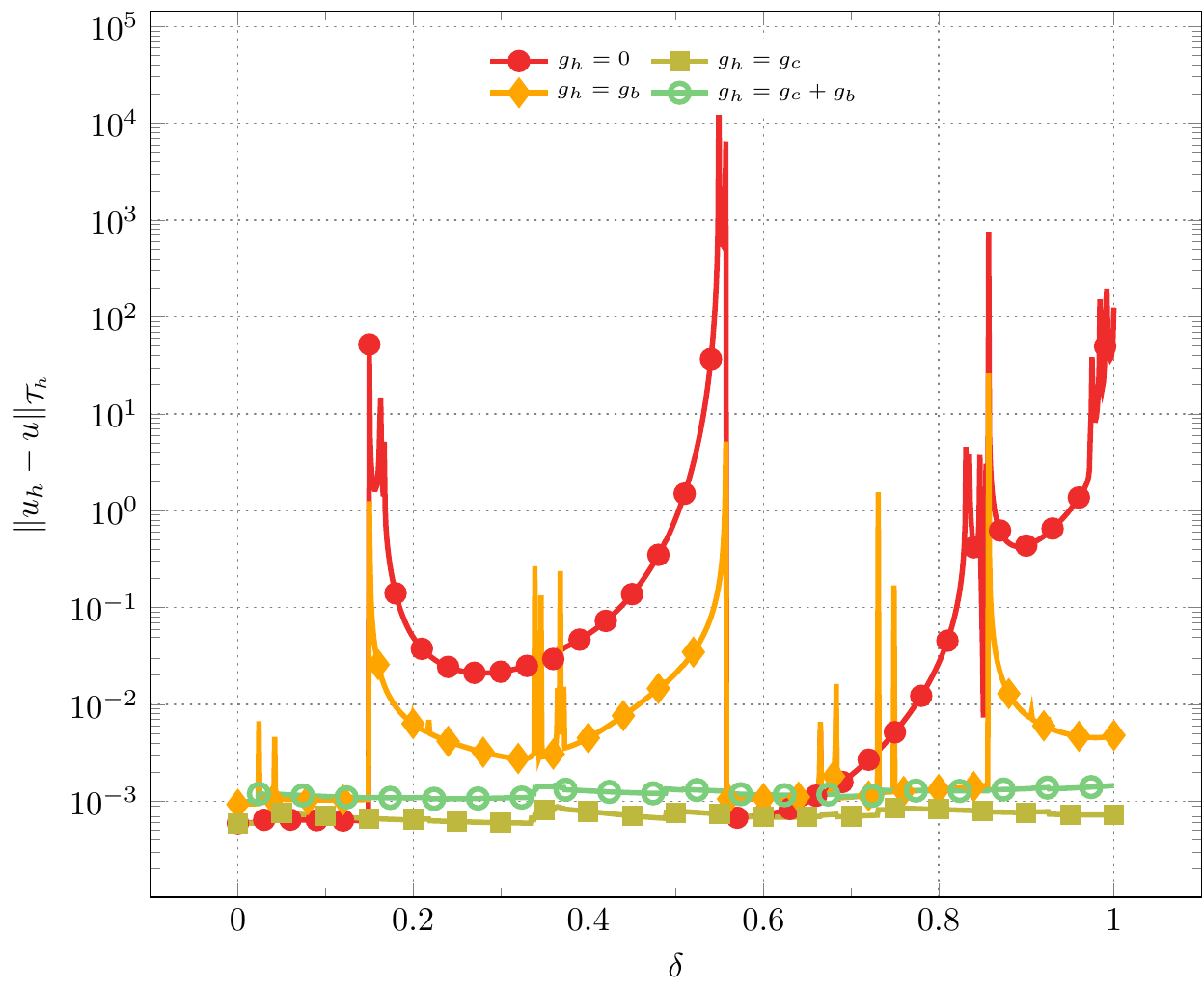}
      \includegraphics[width=0.32\textwidth,page=2]{figures/adr-plot_error_scan_P1.pdf}
      \includegraphics[width=0.32\textwidth,page=1]{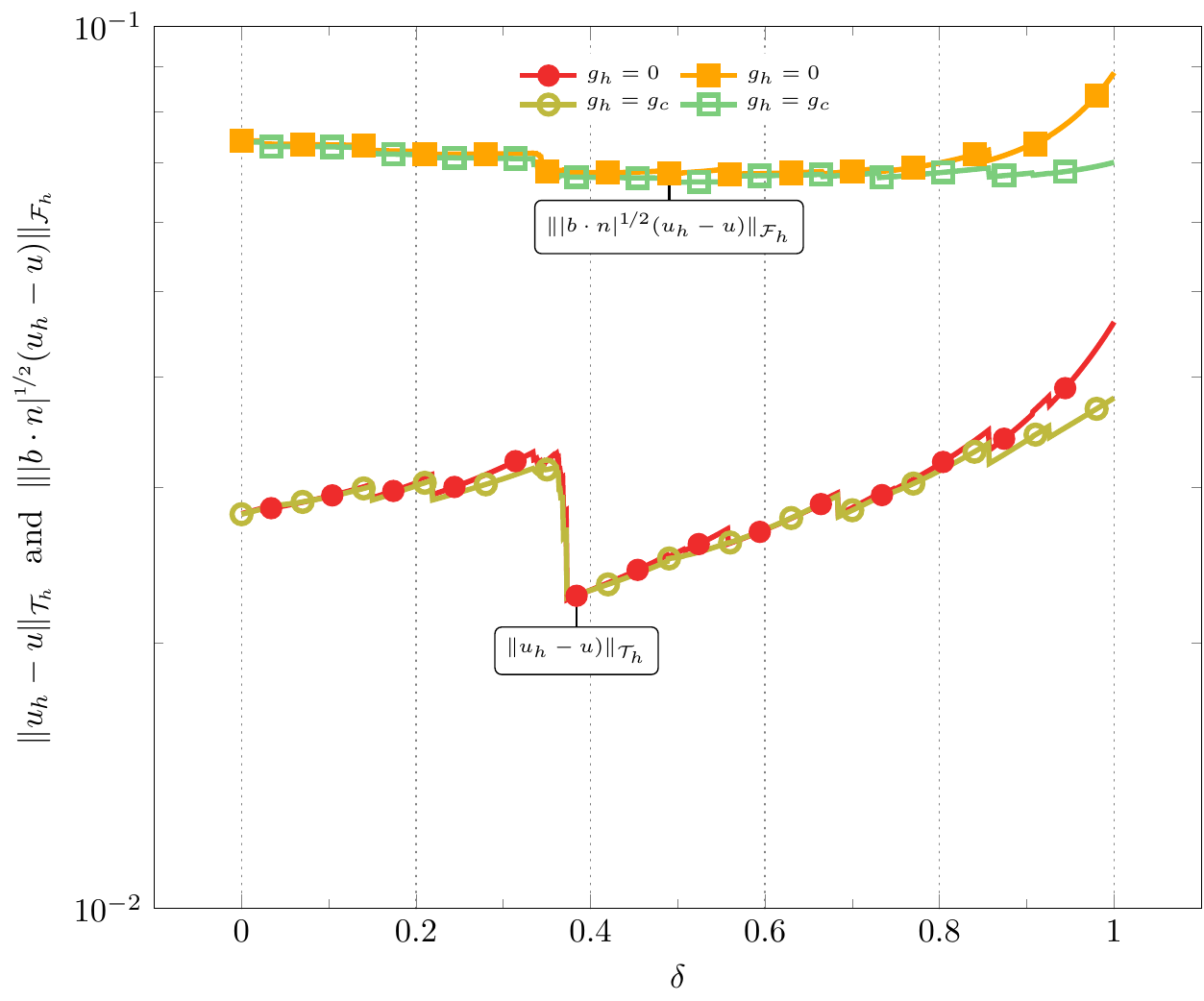}
\end{center}
  \caption{
  Error sensitivity for different choices of $g_h$ in the
  translated domain test considering $\PP_1(\mcT_h)$ and $\PP_0(\mcT_h)$
  spaces for various norms.
  (Left)  $L^2$ error sensitivity for $\PP_1(\mcT_h)$.
  (Middle) Streamline diffusion error sensitivity for $\PP_1(\mcT_h)$.
  (Right) For $\PP_0(\mcT_h)$, the $L^2$ and streamline diffusion error sensitivities
  are nearly unaffected by the absence of $g_c$.
  } \label{fig:error-scan}
\end{figure}
The resulting plots displayed in~Figure~\ref{fig:error-scan} show that both
the $\| \cdot \|_{\Omega}$ and $\|\phi_b^{\onehalf } \bnabla
(\cdot)\|_{\Omega}$ error curves are highly erratic and exhibit large spikes
when the ghost penalties are completely turned off. For the fully stabilized method on
the other hand, the errors are largely insensitive to the translation parameter~$\delta$.
We also observe that the $L^2$ error already appears to be
geometrically robust when only $g_c$ is activated, while the sole activation
of $g_b$ does not have such effect. For the streamline diffusion error, both
$g_b$ and $g_c$ have a stabilizing effect for the particular example.
The stabilizing effect of $g_c$ can be easily explained by recalling the definition of 
the unified ghost penalty $g_h$, cf. (\ref{eq:unified-ghf}),
and realizing that for the given coefficients and considered coarse
mesh size, we have that $c_0 \sim \tfrac{b_c}{h}$.

\subsubsection{Sensitivitity of the conditioning number}
\label{ssec:conditionnumber}
Using the identical setup as in
Section~\ref{ssec:sensitivity-approx-errors}, we now compute the
condition number of the system matrix~(\ref{eq:stiffness-matrix})
as a function of the translation parameter $\delta$. For the fully
stabilized formulation with stabilization parameters $\beta =
\gamma_0 = \gamma_1 = 0.01$, we see that the condition number
changes very mildly with the position parameter $\delta$, while for
the unstabilized formulation, large spikes can be observed up to
the point where the system matrix is practically singular. In a
second run, we study the effect of the magnitude of the ghost
penalty parameters on the magnitude and geometrical robustness of
the condition number. For simplicity, we rescale all ghost penalty
parameters simultaneously. Figure~\ref{fig:diff_sphere_stab}
(right) shows that both the base line magnitude and the fluctuation
of the condition number decreases with increasing size of the
stability parameters with a minimum for some $\gamma_b^f =
\gamma_c^f\in [0.01, 1]$. A further increase of the stability
parameters leaves the condition number insensitive to $\delta$, but
leads to an increase of the overall magnitude. Combined with a
series of convergence experiments (not presented here) for various
parameter choices and combinations, we found that our parameter
choice $\gamma_b^f = \gamma_b^c = 0.01$ offers a good balance
between the accuracy of the numerical scheme and the magnitude and
fluctuation of the condition number.

\begin{figure}[htb]
\begin{center}
    \includegraphics[width=0.48\textwidth,page=1]{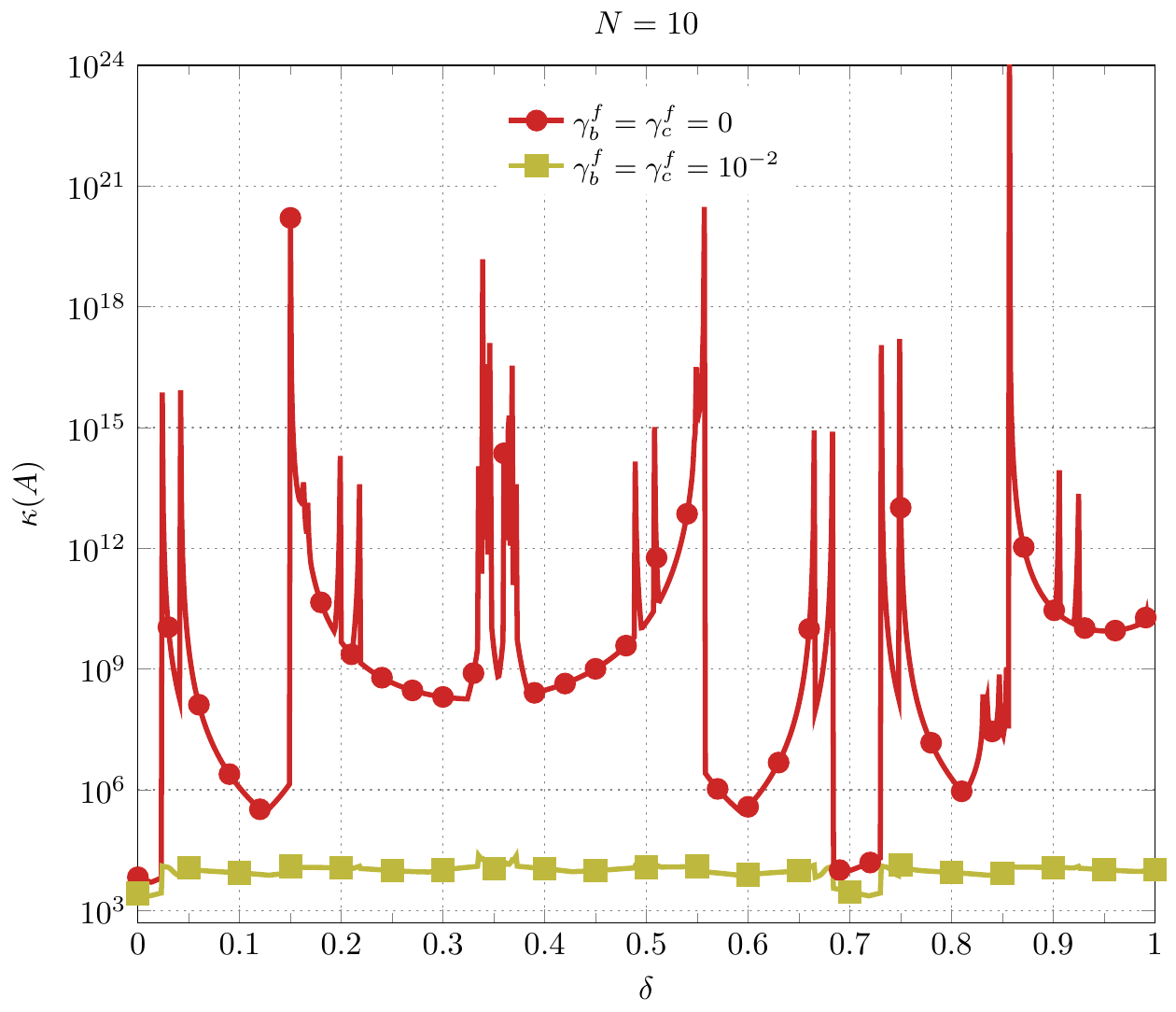} 
    \hspace{0.01\textwidth}
    \includegraphics[width=0.48\textwidth,page=2]{figures/plot_condition_numbers_hyperbolic.pdf}    
    \caption{(Left) Condition number sensitivity for the translated
    domain test with and without ghost penalty
    stabilization~$g_c$. (Right) Condition number sensitivity for
    changing values of the stabilization parameters $\gamma_b^f$ and~$\gamma_c^f$.}
  \label{fig:diff_sphere_stab}
  \end{center}
\end{figure}

\subsection{Time-Dependent Problem}
\label{sec:time-dependent-convergence}
We now briefly demonstrate solving the time-dependent problem~\eqref{eq:time-dep-adr-problem}
using the fully discretized schemes~\eqref{eq:RungeKutta1} and \eqref{eq:RungeKutta3}.
The library deal.II~\cite{dealII91} was used to implement and
conduct the numerical experiments in this section.
Here, the face based stabilization is used for $M_h$ 
(\eqref{eq:stabilized_mass_form}, \eqref{eq:mass_stabilization} and \eqref{eq:gc_face_based})
and the crosswind version of the ghost-penalty \eqref{eq:unified-ghf} is used for~$A_h$.
Let $\Omega$ be a circular domain,
\begin{equation}
\Omega = \{ x \in \RR^2 \st x_1^2 + x_2^2 - 0.25^2 < 0\},
\end{equation}
covered by a quadrilateral background mesh over
$\widetilde{\Omega} = [-0.35,0.35]^2\subset \RR^2$.
Let $b=(0.6,0.8)$ and $c=1$ be constant. 
We consider the manufactured solution
\begin{equation}
u(x,t) = \Theta(x - bt),
\end{equation}
where $\Theta : \RR^2 \to \RR$ is a given function.
This choice of $u$ fulfills \eqref{eq:time-dep-adr-problem} with the data
\begin{align}
f(x,t) & = c\Theta(x-bt), \\
g(x,t) & = \Theta(x-bt).
\end{align}
In the following experiment we choose
\begin{equation}
\Theta(x) = \sum_i \cos(\omega x_i),
\end{equation}
with $\omega = 8\pi$.
We shall consider both $Q_0$ and $Q_2$ elements, i.e., tensor product Lagrange polynomials of order 0 and 2 on quadrilateral elements.
Let $h_s$ denote the side length of the quadrilaterals ($h=\sqrt{2}h_s$).
We solve until end time $T=1$, with a time step 
\begin{equation}
\Delta t = 0.1 h_s,
\label{eq:cfl}
\end{equation}
using the method \eqref{eq:RungeKutta1} for $Q_0$ elements and
\eqref{eq:RungeKutta3} for $Q_2$ elements. A few snapshot of the
numerical solution are shown in
Figure~\ref{fig:time-dependent-snapshots}.

Next, we perform a first convergence study
where the mesh and time step are refined simultaneously.
More detailed numerical experiments of the time-dependent 
case are beyond the scope of the present paper
and will be presented elsewhere.
We are interested in the error at $t=T$ in $L^2$ norm and in the following semi-norm,
\begin{equation}
|v|_b^2 = \frac{1}{2} \int_{\partial \Omega} |b \cdot n| v^2 +  \frac{1}{2} \int_{\mcF_h} |b \cdot n| \jump{v}^2.
\label{eq:b-semi-norm}
\end{equation}
Note that, since we integrate over $\mcF_h$ in \eqref{eq:b-semi-norm} this error 
is computed against the extension of the solution to $\mcT_h$.
The errors for decreasing grid-sizes, $h_s=0.7/40 \cdot 2^{-k}$, are shown in 
Table~\ref{tab:time_dependent_Q0_convergence} for $Q_0$ elements and in
Table~\ref{tab:time_dependent_Q2_convergence} for $Q_2$ elements.
We see that for $Q_0$ elements the error converges as $\mcO(h)$ in $L^2$ norm and
as $\mcO(h^{1/2})$ in the $| \cdot |_b$ semi-norm.
Furthermore, for $Q_2$ elements we see that the error converges at least as $\mcO(h^3)$
in $L^2$ norm and at least as $\mcO(h^{5/2})$ in the $| \cdot |_b$ semi-norm.

In the same manner as in Section~\ref{sec:effectofghostpenalties}, we consider how
the errors change when we perturb the domain.
For a fix grid of size $h_s=0.7/N$, $N=25$, we consider a family of translated domains 
$\Omega_{\delta_k}= \Omega + (\delta_k - \frac{1}{2}) \bft$,
with
$\bft=(h,h)$ and $\delta_k=k/1001$, $k\in \{1,\ldots,1000\}$.
The two error norms are shown in Figure~\ref{fig:time-dependent-error-scan} as a function of $\delta$.
We see that there is only a small variation in the error when we vary $\delta$.

\begin{figure}[htp]
\centerline{\includegraphics[width=.7\paperwidth]{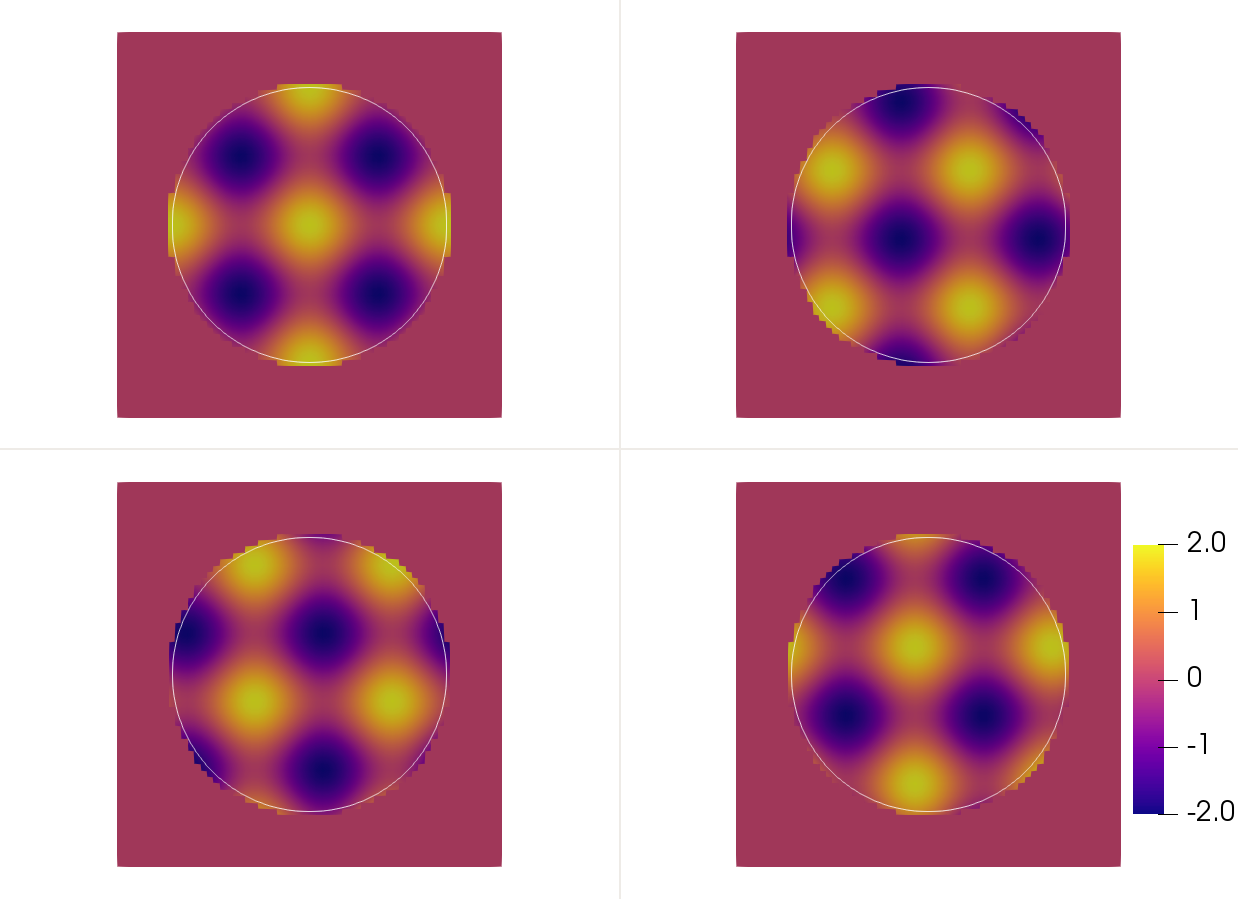}}
  \caption{Snapshots of the solution to the problem from Section~\ref{sec:time-dependent-convergence}, solved with $Q_2$ elements and $h=0.7/160$.
   Top-left: $t=0$, top-right: $t=T/8$, bottom-left: $t=T/4$, bottom-right: $t=3T/8$. 
  }
  \label{fig:time-dependent-snapshots}
\end{figure}

\begin{table}[htb]
  \small
    \caption{Convergence rates for the time-dependent problem in Section~\ref{sec:time-dependent-convergence}.}
\centerline{
\begin{subtable}[t]{.52\textwidth}
	\caption{Using $Q_0$ elements.}
    \label{tab:time_dependent_Q0_convergence}
    \begin {tabular}{ccccc}%
\toprule Level $k$&$\| e_k \|_\Omega $&EOC&$| e_k |_b$&EOC\\\midrule %
\pgfutilensuremath {0}&\pgfutilensuremath {2.00\cdot 10^{-1}}&--&\pgfutilensuremath {0.60}&--\\%
\pgfutilensuremath {1}&\pgfutilensuremath {1.30\cdot 10^{-1}}&\pgfutilensuremath {0.68}&\pgfutilensuremath {0.49}&\pgfutilensuremath {0.27}\\%
\pgfutilensuremath {2}&\pgfutilensuremath {7.17\cdot 10^{-2}}&\pgfutilensuremath {0.81}&\pgfutilensuremath {0.39}&\pgfutilensuremath {0.36}\\%
\pgfutilensuremath {3}&\pgfutilensuremath {3.85\cdot 10^{-2}}&\pgfutilensuremath {0.90}&\pgfutilensuremath {0.29}&\pgfutilensuremath {0.42}\\%
\pgfutilensuremath {4}&\pgfutilensuremath {2.00\cdot 10^{-2}}&\pgfutilensuremath {0.95}&\pgfutilensuremath {0.21}&\pgfutilensuremath {0.46}\\%
\pgfutilensuremath {5}&\pgfutilensuremath {1.02\cdot 10^{-2}}&\pgfutilensuremath {0.97}&\pgfutilensuremath {0.15}&\pgfutilensuremath {0.48}\\\bottomrule %
\end {tabular}%

\end{subtable}
\begin{subtable}[t]{.52\textwidth}
	\caption{Using $Q_2$ elements.}
    \label{tab:time_dependent_Q2_convergence}
    \begin {tabular}{ccccc}%
\toprule Level $k$&$\| e_k \|_\Omega $&EOC&$| e_k |_b$&EOC\\\midrule %
\pgfutilensuremath {0}&\pgfutilensuremath {6.27\cdot 10^{-4}}&--&\pgfutilensuremath {9.38\cdot 10^{-3}}&--\\%
\pgfutilensuremath {1}&\pgfutilensuremath {5.90\cdot 10^{-5}}&\pgfutilensuremath {3.41}&\pgfutilensuremath {1.17\cdot 10^{-3}}&\pgfutilensuremath {3.00}\\%
\pgfutilensuremath {2}&\pgfutilensuremath {6.12\cdot 10^{-6}}&\pgfutilensuremath {3.27}&\pgfutilensuremath {1.67\cdot 10^{-4}}&\pgfutilensuremath {2.81}\\%
\pgfutilensuremath {3}&\pgfutilensuremath {6.29\cdot 10^{-7}}&\pgfutilensuremath {3.28}&\pgfutilensuremath {2.48\cdot 10^{-5}}&\pgfutilensuremath {2.75}\\\bottomrule %
\end {tabular}%

\end{subtable}
}
\end{table}

\begin{figure}[ht!]
  \begin{center}
      \includegraphics[width=0.49\textwidth]{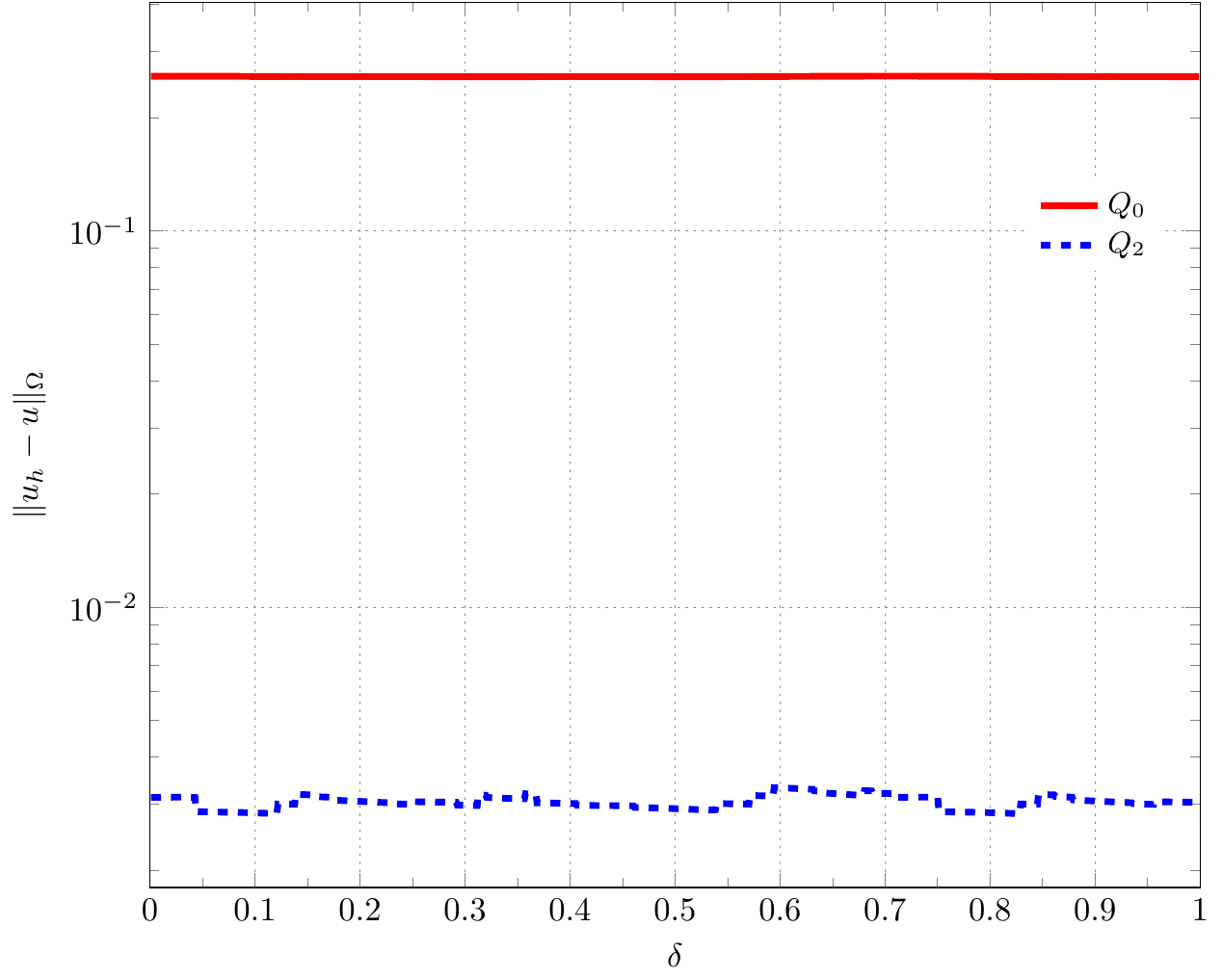}
      \includegraphics[width=0.49\textwidth]{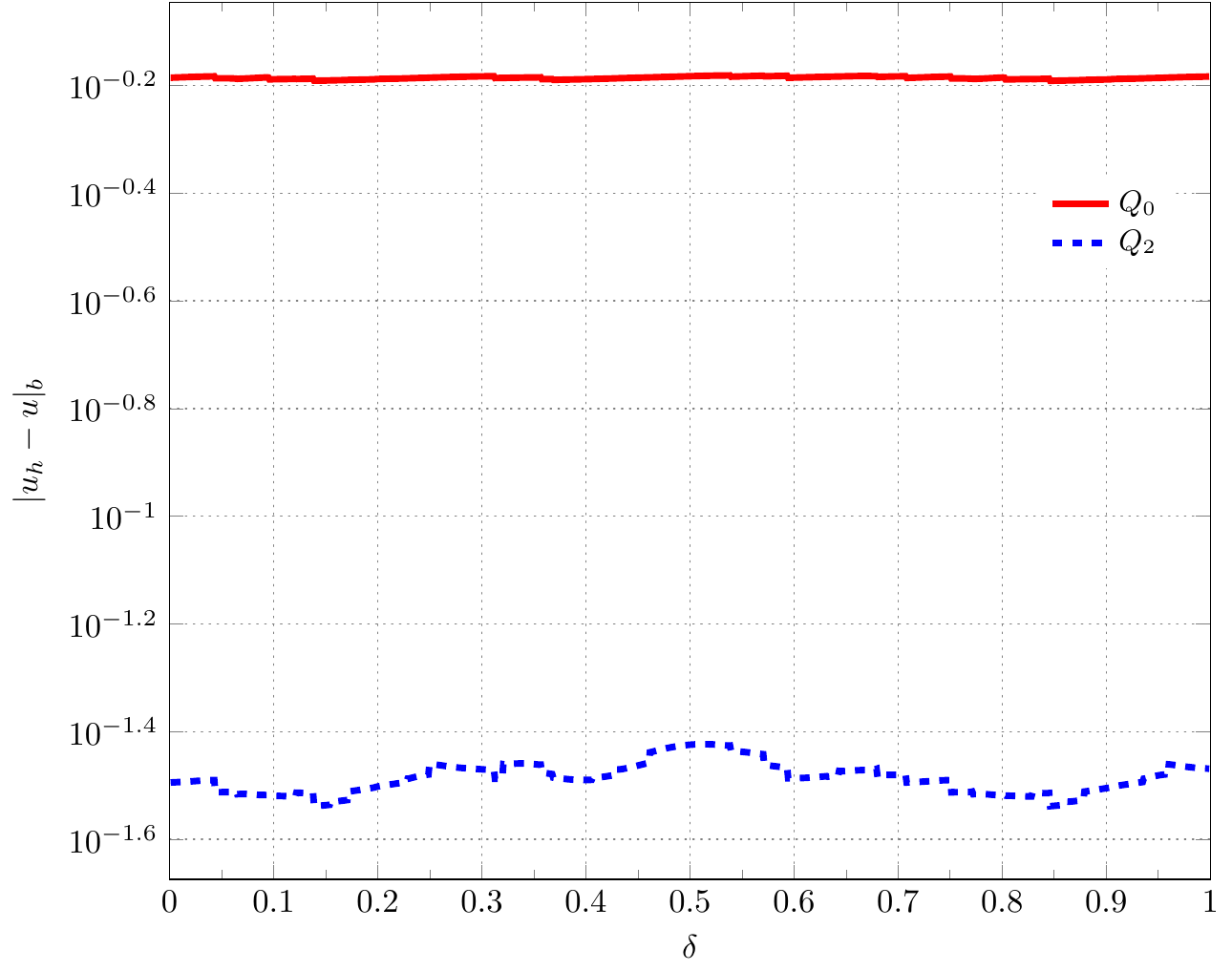}
\end{center}
  \caption{Error sensitivity for the translated domain test, for the time-dependent problem in Section~\ref{sec:time-dependent-convergence}.
  Error in the $L^2$ norm (left) and $|\cdot|_b$ semi-norm (right).} \label{fig:time-dependent-error-scan}
\end{figure}

\section{Conclusions and outlook}
\label{sec:conclusions-outlook}
In this paper we introduced and analyzed stabilized cut discontinuous
Galerkin methods for advection-reaction problems,
with a strong focus on stationary problems.
By adding certain stabilization in the vicinity of 
the embedded boundary, we were able to prove
geometrically robust
optimal a priori error and condition number estimates
which are not affected by the presence of small cut cells.
Our approach works for higher order DG methods
and, avoiding more classical stabilization strategies such as
cell merging, can be incorporated into typical
DG software frameworks with relative ease.
For the sake of simplicity, we assumed simplicial meshes in the 
theoretical analysis. However, similar to \cite{HansboLarsonLarsson2017, } 
the method directly generalizes to quadrilateral or hexahedral 
meshes, which we tested numerically in Section~\ref{sec:time-dependent-convergence}.
While the prototype problems~\eqref{eq:adr-problem} and \eqref{eq:time-dep-adr-problem} are already
of importance in many applications,
our long-term goal is to devise high order CutDG methods
for more complex simulation scenarios,
including mixed-dimensional coupled problems,
time-dependent linear hyperbolic problems and nonlinear
conservations laws on complicated or possibly evolving domains. We
therefore like to conclude our presentation with a short outlook on
natural prototype problems and challenges which need to be
addressed in further developments of our CutDG method.

Regarding time-dependent problems, an additional major challenge
in unfitted discretization methods is that the presence of 
small cut cells can lead to severe time-step
restrictions in explicit time-stepping methods.
To date, all unfitted DG methods and many 
finite volume based Cartesian cut cell methods
address this problem via cell merging.
Devising robust cell merging algorithms can be delicate
for complex and in particular moving 3D domains and requires
non-trivial adaptions of internal data structures.
We think the method~\eqref{eq:time-dep-cutDGM}, where we stabilize the $L^2$ scalar 
product/mass matrix associated with the time-derivative, is a promising alternative route.
The resulting modified mass matrix
is geometrically robust and thus expected to lead
to similar time-step restrictions as its fitted counterparts.
It is part of our ongoing research to combine the presented
stabilized CutDG framework with the general symmetric stabilization
approach proposed in~\cite{BurmanErnFernandez2010}
to devise explicit Runge-Kutta method for first order Friedrich
type operators covering advection reaction problems
as well as linear wave propagation phenomena.

For the numerical discretization of nonlinear scalar
hyperbolic conservation laws
it will be interesting to investigate how
the proposed CutDG stabilization can be combined
with the discontinous Galerkin Runge Kutta methods originally
developed in~\cite{CockburnShu1991,CockburnShu1989,CockburnLinShu1989,CockburnHouShu1990}. 
Here, an important ingredient will be to
combine small cut cell related stabilizations
with limiting techniques to achieve high order
accuracy for smooth solutions and at the same time,
ensure the computation of non-oscillatory, physically compatible
approximations near discontinuities.
This is in particular crucial when the discontinuity 
is located close to the embedded boundary
as ghost penalty stabilizations are designed
to be weakly consistent for \emph{smooth} solutions.
Without modifications,
we expect that such a scheme can lead to non-physical,
entropy-violating solutions, similar to the failure of certain 
linear symmetric stabilization techniques,
see \cite{ErnGuermond2013}. 
For face based ghost penalties,
on possible approach 
might be then to extend the nonlinear weighting of
face based stabilization suggested in \cite{ErnGuermond2013}.
Another possible direction is explore how CutDG stabilizations
can be paired with limiters based on
weighted essentially non-oscillatory
(WENO) methodology which have been proposed for fitted DG methods rather recently, see~\cite{Shu2016} for a short survey and a comprehensive bibliography.

Finally, we want to mention that 
diffusion-convection-reaction problems
in the form of mixed-dimensional coupled surface-bulk problems
have drawn a lot of attention in recent years,
as such problems occur naturally in, e.g.,
flow and transport problems in fractured porous media,
transport of surface active agents in immiscible two-phase flow
systems, and the modelling of cell motility, see, e.g.~\cite{FumagalliScotti2013,HansboLarsonZahedi2016,GanesanTobiska2012,MarthVoigt2013}.
To discretize the resulting multiphysic systems on complex or moving domains,
we plan to combine the presented CutDG method 
with results from  our previous work on CutFEMs for surfaces and
coupled surface-bulk
problems~\cite{BurmanHansboLarsonEtAl2018c,BurmanHansboLarsonEtAl2016a,HansboLarsonMassing2017,Massing2017}.

\section*{Acknowledgments}
The authors wish to thank the anonymous referees for the careful reading of
the manuscript and for the constructive
comments and helpful suggestions, which helped us 
to improve the quality of this paper.

\bibliographystyle{siamplain}
\bibliography{bibliography}

\end{document}